\documentclass{birkmult}

 \newtheorem{theorem}{Theorem}[section]
 \newtheorem{corollary}[theorem]{Corollary}
 \newtheorem{lemma}[theorem]{Lemma}
 \newtheorem{proposition}[theorem]{Proposition}
 \theoremstyle{definition}
 \newtheorem{definition}[theorem]{Definition}
 \theoremstyle{remark}
 \newtheorem{remark}[theorem]{Remark}
 
 \numberwithin{equation}{section}

\newcommand{\C}{\mathbb{C}}
\newcommand{\N}{\mathbb{N}}
\newcommand{\R}{\mathbb{R}}
\newcommand{\T}{\mathbb{T}}
\newcommand{\Z}{\mathbb{Z}}

\newcommand{\cA}{\mathcal{A}}
\newcommand{\cB}{\mathcal{B}}
\newcommand{\cC}{\mathcal{C}}

\newcommand{\cE}{\mathcal{E}}
\newcommand{\cF}{\mathcal{F}}
\newcommand{\cG}{\mathcal{G}}
\newcommand{\cH}{\mathcal{H}}
\newcommand{\cI}{\mathcal{I}}
\newcommand{\cJ}{\mathcal{J}}
\newcommand{\cK}{\mathcal{K}}
\newcommand{\cL}{\mathcal{L}}
\newcommand{\cM}{\mathcal{M}}
\newcommand{\cN}{\mathcal{N}}

\newcommand{\fA}{\mathfrak{A}}
\newcommand{\fC}{\mathfrak{C}}
\newcommand{\fH}{\mathsf{H}}
\newcommand{\fL}{\mathfrak{L}}
\newcommand{\fM}{\mathfrak{M}}
\newcommand{\fN}{\mathsf{N}}
\newcommand{\fW}{\mathsf{W}}

\newcommand{\bA}{\mathbf{A}}
\newcommand{\bB}{\mathbf{B}}
\newcommand{\bC}{\mathbf{C}}
\newcommand{\bD}{\mathbf{D}}
\newcommand{\bE}{\mathbf{E}}
\newcommand{\bF}{\mathbf{F}}
\newcommand{\bG}{\mathbf{G}}
\newcommand{\bI}{\mathbf{I}}
\newcommand{\bJ}{\mathbf{J}}
\newcommand{\bP}{\mathbf{P}}
\newcommand{\bQ}{\mathbf{Q}}
\newcommand{\bX}{\mathbf{X}}
\newcommand{\bU}{\mathbf{U}}
\newcommand{\bV}{\mathbf{V}}
\newcommand{\bW}{\mathbf{W}}
\begin{document}
%
%
%
%
%
%
\title[Finite Section Method for Convolution Type Operators]
{Finite Section Method for a Banach Algebra\\
of Convolution Type Operators on \boldmath{$L^p(\R)$}\\
with Symbols Generated by \boldmath{$PC$} and \boldmath{$SO$}}

\author[Karlovich]{Alexei Yu. Karlovich}
\address{
Departamento de Matem\'atica\\
Faculdade de Ci\^encias e Tecnologia\\
Universidade Nova de Lisboa\\
Quinta da Torre\\
2829--516 Caparica\\
Portugal}
\email{oyk@fct.unl.pt}

\thanks{This work is partially supported by the grant FCT/FEDER/POCTI/MAT/59972/2004.}

\author[Mascarenhas]{Helena Mascarenhas}
\address{
Departamento de Mathem\'atica\\
Instituto Superior T\'ecnico\\
Av. Rovisco Pais\\
1049-001 Lisboa \\
Portugal}
\email{hmasc@math.ist.utl.pt}

\author[Santos]{Pedro A. Santos}
\address{
Departamento de Mathem\'atica\\
Instituto Superior T\'ecnico\\
Av. Rovisco Pais\\
1049-001 Lisboa \\
Portugal}
\email{pasantos@math.ist.utl.pt}
\subjclass{Primary 65R20; Secondary 45E10, 47B35, 47L80, 65J10}

\keywords{%
Finite section method,
slowly oscillating function,
Fourier multiplier,
algebraization,
essentialization,
homogenization,
local principle,
two projections theorem}


\begin{abstract}
We prove the applicability of the finite section method to an arbitrary
operator in the Banach algebra generated by the operators of multiplication
by piecewise continuous functions and the convolution operators with symbols
in the algebra generated by piecewise continuous and slowly oscillating
Fourier multipliers on $L^p(\R)$, $1<p<\infty$.
\end{abstract}

\maketitle
\section{Introduction}
Given $1<p<\infty$, let $\cB:=\cB(L^p(\R))$ denote the Banach
algebra of all bounded linear operators on the Lebesgue space
$L^p(\R)$. Let $[PC,SO]$ be the smallest $C^*$-subalgebra of
$L^\infty(\R)$ containing all piecewise continuous ($PC$) and
slowly oscillating ($SO$) functions, and let $[PC_p,SO_p]$ stand
for its Fourier multiplier analogue, which is a Banach subalgebra
of $\cM_p$, the Banach algebra of all Fourier multipliers on
$L^p(\R)$. The Fredholm theory for the smallest Banach subalgebra
of $\cB(L^p(\R))$ that contains all the convolution type operators
$aF^{-1}bF$ where $F$ is the Fourier transform given by
\begin{equation}\label{eq:Fourier}
(F\varphi)(x):=\frac{1}{\sqrt{2\pi}}\int_\R e^{ixy}\varphi(y)\,dy
\quad(x\in\R)
\end{equation}
and $a\in[PC,SO]$, $b\in[PC_p,SO_p]$ is constructed in \cite{BBK04-OT,BBK04-MN}
(in those papers, even matrix functions $a$ and $b$ are considered).

Let $\R_+:=(0,\infty)$ and $\R_-:=(-\infty,0)$. For $\tau\in\R_+$, consider
the operators
\[
(P_\tau \varphi)(t)
:=
\left\{
\begin{array}{llc}
\varphi(t) & \mbox{if} & |t| < \tau, \\
0    & \mbox{if} & |t| > \tau,
\end{array}
\right.
\quad
Q_\tau := I - P_\tau
\]
acting on $L^p(\R)$ with norm $1$. Clearly, $P_\tau\to I$ and
$Q_\tau\to 0$ strongly as $\tau\to\infty$. One says that
\textit{the finite section method applies} to an operator
$A\in\cB(L^p(\R))$ if there exists a positive constant $\tau_0$,
such that for any $\tau>\tau_0$ and any $f\in L^p (\R)$ there
exists a unique solution $\varphi_\tau$ of the equation
\begin{equation}\label{eq:FSM}
A_\tau\varphi_\tau:=(P_\tau A P_\tau + Q_\tau) \varphi_\tau= f
\quad (\tau\in\R_+)
\end{equation}
and $\varphi_\tau$ converges in the norm of $L^p(\R)$ to a solution of the
equation $A\varphi = f$ as $\tau\to\infty$.

We refer to the monographs by Gohberg and Feldman \cite{GF71},
Pr\"ossdorf and Silbermann \cite{PS91}, Hagen, Roch, and
Silbermann \cite{HRS95,HRS01} for a general theory of projection
methods as well as for more specific issues of the finite section
method for convolution type operators and algebras generated by
them.

Roch, Silbermann, and one of the authors \cite{RSS97} studied the
applicability of the finite section method to an operator in the
smallest $C^*$-subalgebra of $\cB(L^2(\R))$ generated by all $aI$ with
$a\in PC$, $W^0(b)$ with $b\in PC=PC_2$, and the so-called flip
operator. For general $p\ne 2$, they proved in \cite{RSS09} the
applicability of the finite section method for an arbitrary
operator in the Banach algebra generated by the operators of
multiplication by piecewise continuous functions $(PC)$ and by the
convolution operators with piecewise continuous Fourier
multipliers $(PC_p)$.

The aim of this paper is to take one more step forward. We prove
the applicability of the finite section for an arbitrary operator $A$
in the Banach subalgebra of $\cB(L^p(\R))$, $1<p<\infty$,
generated by all operators $aI$ of multiplication by functions
$a\in PC$ and by all Fourier convolution operators
\[
W^0(b):=F^{-1}bF
\]
with $b\in [PC_p,SO_p]$.

Our approach to analyze the applicability of the finite sections
method will follow a general scheme to treat approximation
problems. This scheme goes back to Silbermann \cite{Silbermann81}
(see also \cite[Section~1.6]{HRS95}). It can be summarized as
follows. Let $\cA$ be a set of (generalized) sequences that
contains all sequences of the form
\begin{equation}\label{eq:sequence-FSM}
(A_\tau)=(P_\tau AP_\tau+Q_\tau)\quad(\tau\in\R_+).
\end{equation}

\begin{enumerate}
\item
{\bf Algebraization}: Find a unital Banach algebra $\cE$
containing $\cA$ and a closed ideal $\cG$ of $\cE$ such that the
original problem becomes equivalent to an invertibility problem in
the quotient algebra $\cE/\cG$.

\item
{\bf Essentialization}: Find a unital inverse closed subalgebra $\cF$ of $\cE$
that contains $\cA$ and a closed ideal $\cJ$ of $\cF$ that
contains $\cG$, such that $\cJ$ can be lifted. The latter means
that one has full control about the difference between the
invertibility of a coset of a sequence $(A_\tau) \in \cF$ in the
algebra $\cF/\cG$ and the invertibility of the coset of the same
sequence in $\cF/\cJ$. This control is usually guaranteed by a
lifting theorem.

\item
{\bf Localization}: Find a unital subalgebra $\cL$ of $\cF$ such
that
\begin{enumerate}
\item $\cA,\cJ\subset\cL$;
\item
$\cL/\cJ$ is inverse closed in $\cF/\cJ$;
\item
the quotient algebra $\cL/\cJ$ has a large center.
\end{enumerate}
Use a local principle to translate the invertibility problem in
the algebra $\cL/\cJ$ to a family of simpler invertibility
problems in the local algebras.

\item
{\bf Identification}: Find conditions for the invertibility of the
cosets of sequences in $\cA$ in the local algebras.
\end{enumerate}

The paper is organized as follows.  Section~\ref{section:PC-SO}
contains all necessary properties of piecewise continuous and
slowly oscillating Fourier multipliers. In
Section~\ref{section:SIO}, we formulate several auxiliary results
on singular integral operators with piecewise constant
coefficients, which will be used in
Sections~\ref{section:identification-3} and \ref{section:main}. In
Section~\ref{section:algebraization}, the algebra $\cE$ and its
ideal $\cG$ are defined and Kozak's theorem is formulated.

Let $\cA$ denote the smallest closed subalgebra of $\cE$ that
contains the constant sequences $(aI)$ with $a\in PC$ and
$(W^0(b))$ with $b\in[PC_p,SO_p]$ and the sequence $(P_\tau)$. In
Section~\ref{section:essentialization}, we perform the
essentialization step: we introduce the algebra $\cF\subset\cE$,
its ideal $\cJ$, and the homomorphisms $\fW_i$,
$i\in\{-1,0,1\}$. It is shown that $\cA\subset\cF$. The main
result of Section~\ref{section:essentialization}
(Theorem~\ref{th:lifting-our}) says that for a sequence
$\bA=(A_\tau)\in\cF$ the coset $\bA+\cG$ is invertible in
$\cF/\cG$ if and only if the operators $\fW_i(\bA)$,
$i\in\{-1,0,1\}$, are invertible and the coset $\bA+\cJ$ is
invertible in $\cF/\cJ$. The algebra $\cF/\cJ$ is still too large for
effective studying.

In Section~\ref{section:localization}, we introduce the algebra
$\cL$ of sequences of local type such that
$\cA\subset\cL\subset\cF$ and the algebra $\cL/\cJ$ has a large
center. So the latter algebra can be studied with the aid of the
Allan-Douglas local principle. According to it, the invertibility
of $\bA+\cJ\in\cL/\cJ$ is equivalent to the invertibility of the
local representatives $\Phi_{\xi,\eta}^\cJ(\bA)$ in the local
algebras $\cL_{\xi,\eta}^\cJ$, where
\[
(\xi,\eta)\in
\big(\R\times M_\infty(SO)\big)\cup
\big(\{\infty\}\times\R\big)\cup
\big(\{\infty\}\times M_\infty(SO)\big)
\]
and $M_\infty(SO)$ is the fiber of the maximal ideal space of $SO$
over the point $\infty$ (see Section~\ref{subsection:fibration}).

In Section~\ref{section:homogenization}, we introduce the
homomorphisms $\fH_\eta$ for $\eta\in\R$ and show that if a
sequence $\bA\in\cA$ is stable, then all operators
$\fH_\eta(\bA)$ for $\eta\in\R$ are invertible.

In
Sections~\ref{section:identification-1}--\ref{section:identification-3},
we study the invertibility in the local algebras
$\cL_{\xi,\eta}^\cJ$. It turns out that these algebras are too
large for a complete description, so we will restrict ourselves to
studying the invertibility in their subalgebras
$\cA_{\xi,\eta}^\cJ:=\Phi_{\xi,\eta}^\cJ(\cA)$.

Let $\bA\in\cA$. In Section~\ref{section:identification-1}, we
prove that the invertibility of the operator $\fW_0(\bA)$ is
sufficient for the invertibility of $\Phi_{\xi,\eta}^\cJ(\bA)$ in
the local algebra $\cL_{\xi,\eta}^\cJ$ for $(\xi,\eta)\in\R\times
M_\infty(SO)$. Further, in Section~\ref{section:identification-2},
we obtain that the invertibility of the operator $\fH_\eta(\bA)$
is sufficient for the invertibility of
$\Phi_{\infty,\eta}^\cJ(\bA)$ in the local algebra
$\cL_{\infty,\eta}^\cJ$ with $\eta\in\R$.

Section~\ref{section:identification-3} is devoted to
$\cA_{\infty,\eta}^\cJ$ and $\cL_{\infty,\eta}^\cJ$ with $\eta\in
M_\infty(SO)$. We show that the invertibility problem in these
algebras can be reduced to the invertibility problem in the pairs
of simpler algebras $\cA_\eta^\pm$ and $\cL_\eta^\pm$. It turns
out that the algebras $\cA_\eta^\pm$ are generated by two
idempotents and the identities of the algebras $\cL_\eta^\pm$.
Applying the two-idempotents theorem
(Theorem~\ref{th:two-idempotents}), we get necessary and
sufficient conditions for the invertibility of an element of
$\cA_\eta^\pm$ in the algebra $\cL_\eta^\pm$. These results lead
to a criterion of the invertibility of an element of
$\cA_{\infty,\eta}^\cJ$ in the algebra $\cL_{\infty,\eta}^\cJ$.

In Section~\ref{section:main} we gather the results obtained in
Sections~\ref{section:localization}--\ref{section:identification-3}
and prove our main result: a criterion for the stability of a
sequence $\bA=(A_\tau)\in\cA$. Further, it is specified for the
sequences of the finite section method $(P_\tau AP_\tau+Q_\tau)$.
Finally we illustrate our results for the paired convolution
operators
\[
A=W^0(a)\chi_-I+W^0(b)\chi_+I
\]
with $a,b\in [PC_p,SO_p]$, where $\chi_-$ and $\chi_+$ are the
characteristic functions of the half-axes $(-\infty,0)$ and
$(0,+\infty)$.
\section{Piecewise continuous and slowly oscillating Fourier multipliers}
\label{section:PC-SO}
\subsection{Function algebras}
Let $1\le p\le\infty$ and $L^p(\R)$ denote the usual Lebesgue
space on $\R=(-\infty,+\infty)$ with the standard norm denoted by
$\|\cdot\|_p$, let $C(\overline{\R})$ and $C(\dot{\R})$ denote the
spaces of continuous functions on
$\overline{\R}=[-\infty,+\infty]$ and $\dot{\R}=\R\cup\{\infty\}$,
respectively;
\[
C_0(\R):=\big\{f\in C(\dot{\R})\ : \ f(\infty)=0\big\},
\quad
C_b(\R):=C(\R)\cap L^\infty(\R);
\]
and $PC$ stand for the set of all functions $f:\dot{\R}\to\C$
which possess a finite left-hand limit $f(x^-)$ and a finite
right-hand limit $f(x^+)$ at every point $x\in\dot{\R}$.

Let $V_1(\R)$ be the set of functions $a:\overline{\R}\to\C$ with the finite
total variation
\[
V_1(a):=\sup\left\{\sum_{i=1}^n|a(x_i)-a(x_{i-1})| :\
-\infty\le x_0< x_1<\dots<x_n\le+\infty, n\in\N\right\}.
\]
where the supremum is taken over all finite decompositions of the
real line $\R$. It is well known that $V_1(\R)$ is a Banach
algebra under the norm
\[
\|a\|_{V_1(\R)}:=\|a\|_\infty+V_1(a).
\]

For a continuous function $f:\R\to\C$ and a set $I\subset\R$, let
\[
\operatorname{osc}(f,I):=\sup_{t,s\in I}|f(t)-f(s)|.
\]
Following Power \cite{Power80} we denote by $SO$ the set of all
\textit{slowly oscillating functions},
\[
SO:=\left\{f\in C_b(\R)\ :\ \lim_{x\to+\infty}
\operatorname{osc}(f,[-2x,-x]\cup[x,2x])=0\right\}.
\]
Clearly, $SO$ is a $C^*$-subalgebra of $L^\infty(\R)$ and $C(\dot{\R})\subset SO$.

Let $C_b^1(\R)$ consist of all functions $a\in C_b(\R)$ with $a'\in C_b(\R)$.
The following lemma is, in fact, proved in \cite[pp.~154--155]{BBK04-OT}.
\begin{lemma}\label{le:tilde-SO1}
If a function $a$ is even and lies in
\[
\widetilde{SO}^1:=\left\{f\in C_b^1(\R):\lim_{x\to\infty}xf'(x)=0\right\},
\]
then $a\in SO$.
\end{lemma}
However, $\widetilde{SO}^1$ is not contained in $SO$ because its elements can
have different slowly oscillating behavior at $-\infty$ and $+\infty$.
\begin{lemma}[{\cite[Corollary~2.4]{BBK04-OT}}]
\label{le:SO1}
Every function $a\in SO$ can be uniformly approximated by
functions in the non-closed algebra
\begin{equation}\label{eq:def-SO1}
SO^1:=\widetilde{SO}^1\cap SO.
\end{equation}
\end{lemma}
From now on, we consider $1<p<\infty$. Let $\cB:=\cB(L^p(\R))$ be
the Banach algebra of all bounded linear operators on $L^p(\R)$
and $\cK:=\cK(L^p(\R))$ be the closed two-sided ideal of all
compact operators on $L^p(\R)$. The Cauchy singular integral
operator $S_\R$ is defined on $L^p(\R)$ by
\[
(S_\R f)(x):=\frac{1}{\pi i}\int_\R \frac{f(t)}{t-x}\,dt
\quad(x\in\R)
\]
where the integral is understood in the sense of principal value.

Let $F:L^2(\R)\to L^2(\R)$ denote the Fourier transform given by
\eqref{eq:Fourier}, and let $F^{-1}$ be its inverse.
We define the operator $W^0(a)$ on $L^2(\R)\cap L^p(\R)$ by
\begin{equation}\label{eq:def-operator-W0}
(W^0(a)\varphi)(x):=(F^{-1}aF\varphi)(x),
\quad
\varphi\in L^2(\R)\cap L^p(\R).
\end{equation}
A function $a\in L^\infty(\R)$ is called a \textit{Fourier
multiplier on} $L^p(\R)$ if the operator $W^0(a)$ given by
\eqref{eq:def-operator-W0} can be extended to a bounded linear
operator on $L^p(\R)$, which will again be denoted by $W^0(a)$. By
$\chi_+$ (resp. $\chi_-$) denote the characteristic function of
the semi-axis $\R_+:=(0,+\infty)$ (resp. $\R_-:=(-\infty,0)$).
Then
\[
W^0(\chi_-)=P_\R:=(I+ S_\R)/2,
\quad
W^0(\chi_+)=Q_\R:=(I-S_\R)/2
\]
are two complementary projections on $L^p(\R)$
(see e.g. \cite[Section~2]{Duduchava79} or \cite[Section~2.5]{BKS02}).

The set $\cM_p$ of all Fourier multipliers
on $L^p(\R)$ is defined as
\[
\cM_p:=\big\{a\in L^\infty(\R)\ :\ W^0(a)\in\cB(L^p(\R))\big\}.
\]
It is well known that $\cM_p$ is a Banach algebra with the norm
\[
\|a\|_{\cM_p}:=\|W^0(a)\|_{\cB(L^p(\R))},
\]
and
\begin{equation}\label{eq:multipliers-embedding}
\cM_p\subset\cM_2=L^\infty(\R)
\quad\text{for all}\quad p\in(1,\infty).
\end{equation}
According to Stechkin's inequality (see e.g. \cite[Theorem~17.1]{BKS02}),
$\cM_p$ contains all functions $a\in PC$ of finite total variation and
\[
\|a\|_{\cM_p}\le \|S_\R\|_{\cB}\big(\|a\|_\infty+V_1(a)\big).
\]
Let $C_p(\dot{\R})$ and $C_p(\overline{\R})$ stand for the closure in
$\cM_p$ of the set of all functions with finite total variation in
$C(\dot{\R})$ and $C(\overline{\R})$, respectively. Denote by
$PC_p$ the closure in $\cM_p$ of the set of all piecewise constant
functions on $\R$ which have at most finite sets of jumps.
\begin{lemma}[{\cite[Lemma~1.1]{SimMin86}}]
\label{le:Simonenko}
If $1<p<\infty$, then the algebra $C_p(\overline{\R})$ is generated
by the functions $f(x)=1$ and $g(x)=\tanh x$.
\end{lemma}
In view of \eqref{eq:def-SO1}, every function $a\in SO^1$ has the
properties $a\in C_b(\R)$ and $Da\in C_b(\R)$ where
$(Da)(x)=xa'(x)$. Therefore, by Mikhlin's theorem (see e.g.
\cite[Theorem~5.2.7(a)]{Grafakos08}), every function $a\in SO^1$
belongs to all spaces $\cM_p$ for $p\in(1,\infty)$. Hence, by
analogy with $C_p(\dot{\R})$ and taking into account
Lemma~\ref{le:SO1}, one can define the set $SO_p$ of slowly
oscillating Fourier multipliers as the closure in $\cM_p$ of the
set $SO^1$. Clearly, $SO_p$ is a Banach subalgebra of $\cM_p$. The
set $SO_p$ was introduced in \cite{BBK04-OT,BBK04-MN}.
From
\eqref{eq:multipliers-embedding} it follows that
\begin{equation}\label{eq:SO1-embedded-in-SO}
SO^1\subset SO_p\subset SO_2=SO
\quad\text{for all}\quad 1<p<\infty.
\end{equation}

For $a\in L^\infty(\R)$ by $\overline{a}$ denote the function
$\overline{a}(t):=\overline{a(t)}$, where the bar denotes the
complex conjugation.
\begin{lemma}\label{le:SOq}
Let $1<p<\infty$ and $1/p+1/q=1$. If $a\in SO_p$, then $\overline{a}\in SO_q$.
\end{lemma}
\begin{proof}
Let $a$ be the limit of a sequence $(a_n)_{n=1}^\infty\subset
SO^1$ in the norm of $\cM_p$. Obviously, $\overline{a_n}\in
SO^1\subset SO_q$. Since $a-a_n\in\cM_p$, from
\cite[Section~9.3(b)]{BS06} we see that
$\overline{a}-\overline{a_n}\in\cM_q$ and
$\|\overline{a}-\overline{a_n}\|_{\cM_q}=\|a-a_n\|_{\cM_p}\to 0$
as $n\to\infty$. This means that $\overline{a}\in SO_q$.
\end{proof}
We denote by $[PC,SO]$ the smallest $C^*$-subalgebra of $L^\infty(\R)$ that
contains $PC$ and $SO$. Similarly, for every $p\in(1,\infty)$, we introduce
the Banach subalgebra $[PC_p,SO_p]$ of $\cM_p$ generated by $PC_p$ and $SO_p$.
Obviously,
\[
[PC_p,SO_p]\subset[PC_2,SO_2]=[PC,SO]
\quad\text{for all}\quad 1<p<\infty.
\]
\subsection{Compactness of Hankel type operators and commutators}
Put
\[
\C_+:=\{z\in\C:\operatorname{Im}z>0\},
\quad
\C_-:=\{z\in\C:\operatorname{Im}z<0\},
\]
and denote by $H^\infty(\C_\pm)$ the set of all bounded and analytic functions
in $\C_\pm$. Fatou's theorem says that functions in $H^\infty(\C_\pm)$ have
non-tangential limits on $\R$ almost everywhere, and we denote by $H_\pm^\infty$
the set of all $a\in L^\infty(\R)$ that are non-tangential limits of functions
in $H^\infty(\C_\pm)$. It is well known that $H_\pm^\infty$ are closed
subalgebras of $L^\infty(\R)$. Sarason discovered (see e.g. \cite[p.~107]{BS06})
that the smallest closed subalgebra of $L^\infty(\R)$ that contains both
$C(\dot{\R})$ and $H_\pm^\infty$ coincides with the sum of $C(\dot{\R})$ and
$H_\pm^\infty$.
\begin{lemma}\label{le:compactness-Hankel}
Let $1<p<\infty$. If $b\in SO_p$, then $\chi_+W^0(b)\chi_-I$ and
$\chi_-W^0(b)\chi_+I$ are compact on $L^p(\R)$.
\end{lemma}
\begin{proof}
From \cite[Sections~9.11, 9.35, and 2.80]{BS06} it follows that
\begin{equation}\label{eq:SO-embedded-in-QC}
SO\subset (C(\dot{\R})+H_+^\infty)\cap(C(\dot{\R})+H_-^\infty).
\end{equation}
By Hartman's theorem (see e.g. \cite[Theorem~2.18]{BKS02} and also
\cite[Section~2.54]{BS06}),
if $b\in C(\dot{\R})+H_-^\infty$ (resp. $b\in C(\dot{\R})+H_+^\infty$),
then the operator $\chi_+W^0(b)\chi_-I$ (resp. $\chi_-W^0(b)\chi_+I$)
is compact on $L^2(\R)$.

Let $b\in SO_p$ be the limit in the norm of $\cM_p$ of a sequence $b_n \in SO^1$.
From the above results and
\eqref{eq:SO1-embedded-in-SO}--\eqref{eq:SO-embedded-in-QC} it follows that
$\chi_\pm W^0(b_n)\chi_\mp I$ are compact on $L^2(\R)$. Moreover, these
operators are bounded on every $L^p(\R)$ with $p\in(1,\infty)$.
By the Krasnosel'skii interpolation theorem \cite{Krasnoselskii60}
(see also \cite[Theorem~3.10]{KZPS76}),
the operators $\chi_\pm W^0(b_n)\chi_\mp I$ are compact on $L^p(\R)$
for every $p\in(1,\infty)$. From
\[
\begin{split}
\|\chi_\pm W^0(b)\chi_\mp I-\chi_\pm W^0(b_n)\chi_\mp I\|_{\cB(L^p(\R))}
&\le
\|W^0(b-b_n)\|_{\cB(L^p(\R))}
\\
&
=\|b-b_n\|_{\cM_p}=o(1)
\end{split}
\]
as $n\to\infty$ it follows that the operators $\chi_\pm W^0(b)\chi_\mp I$ are
compact on $L^p(\R)$.
\end{proof}
\begin{lemma}[{\cite[Lemma~7.4]{Duduchava79}}]
\label{le:Duduchava}
Let $1<p<\infty$. If $a\in C(\overline{\R})$ and $b\in C_p(\overline{\R})$,
then the operator $aW^0(b)-W^0(b)aI$ is compact on the space $L^p(\R)$.
\end{lemma}
\begin{lemma}[{\cite[Theorem~4.2]{BBK04-OT}}]
\label{le:compactness-commutator}
Let $1<p<\infty$. If $a\in [PC,SO]$, $b\in SO_p$ or $a\in SO$, $b\in [PC_p,SO_p]$,
then the operator $aW^0(b)-W^0(b)aI$ is compact on the space $L^p(\R)$.
\end{lemma}
\subsection{Maximal ideal spaces of some commutative Banach algebras}
\label{subsection:fibration}
Let $M(SO)$ denote the maximal ideal space of $SO$.
Identifying points $t\in\dot{\R}$ with the evaluation functionals
at $t$, one can identify the fiber of $M(SO)$ over $t\in\dot{\R}$
by
\[
M_t(SO)=\big\{\eta\in M(SO)\ :\ \eta|_{C(\dot{\R})}=t\big\}.
\]
If $t\in\R$, then the fiber $M_t(SO)$ consists of the only evaluation functional
at $t$, and thus
\[
M(SO)=\bigcup_{t\in\dot{\R}}M_t(SO)=\R\cup M_\infty(SO).
\]
According to \cite[Proposition~3.1]{BBK04-OT}, the fiber $M_\infty(SO)$ has
the form
\[
M_\infty(SO)=(\operatorname{clos}_{SO^*}\R)\setminus\R
\]
where $\operatorname{clos}_{SO^*}\R$ is the weak-star closure of
$\R$ in $SO^*$, the dual space of $SO$. Thus, for every functional
$\eta\in M_\infty(SO)$ there exists a net $t_\omega\in\R$ that
tends to $\infty$ in the usual topology of $\R$ and such that
$\eta$ is the limit of $t_\omega$ in the Gelfand topology, that
is, $\eta(a)=\lim\limits_\omega a(t_\omega)$ for every $a\in SO$.
The next lemma shows, in particular, that for a fixed $a\in SO$,
the net $t_\omega$ can be replaced by a sequence
$\tau_n\to+\infty$.
\begin{lemma}[{\cite[Corollary~3.3]{BBK04-OT}}]\label{le:SO-partial-limits}
If $(a_k)_{k=1}^\infty$ is a countable subset of $SO$ and $\eta$
is an element of $M_\infty(SO)$, then there exists a sequence
$(\tau_n)_{n=1}^\infty\subset\R_+$ such that $\tau_n>1$,
$\tau_n\to+\infty$ as $n\to\infty$, and for every
$x\in\R\setminus\{0\}$,
\[
\eta(a_k)=\lim_{n\to\infty}a_k(\tau_nx) \quad (k\in\N).
\]
\end{lemma}
In what follows we write
\[
a(\eta):=\eta(a)
\]
for every $a\in SO$ and every $\eta\in M(SO)$.

A unital Banach algebra $A$ is called \textit{inverse closed} in a unital
Banach algebra $B\supset A$ with the same unit if for any $a\in A$ invertible
in $B$ its inverse $a^{-1}$ belongs to $A$.

From \cite[Section~3]{BBK04-OT} it follows that the Banach algebras $SO_p$
and $[PC_p,SO_p]$ are inverse closed in the $C^*$-algebras $SO$ and $[PC,SO]$,
respectively, and their maximal ideal spaces coincide as sets:
\[
M(SO_p)=M(SO),
\quad
M([PC_p,SO_p])=M([PC,SO]).
\]
It is well known that $M_\infty(PC)=\{\pm\infty\}$. According to
\cite[Section~1]{Power80}, the fiber $M_\infty([PC,SO])$ is homeomorphic to
the product
\[
M_\infty(PC)\times M_\infty(SO)=\{\pm\infty\}\times M_\infty(SO)
\]
and the homeomorphism is given by the restriction map
$\beta\mapsto(\beta|_{PC},\beta|_{SO})$. Thus, every $\beta\in
M_\infty([PC,SO])$ can be viewed as a functional of the form
either $(+\infty,\eta)$ or $(-\infty,\eta)$ with $\eta\in
M_\infty(SO)$. Fix $\eta\in M_\infty(SO)$. Then there exists a
homomorphism
\begin{equation}\label{eq:functional-alpha-PC-SO}
\alpha_\eta:[PC,SO]\to PC|_{\{\pm\infty\}}, \quad
(\alpha_\eta a)(\pm\infty)=a_\eta(\pm\infty):=(\pm\infty,\eta)a.
\end{equation}
As the values $a_\eta(\pm\infty)$ are uniquely defined, for every
$\eta\in M_\infty(SO)$ we get the homomorphism
\begin{equation}\label{eq:functional-gamma-PC-SO}
\gamma_\eta:[PC,SO]\to PC, \quad
\gamma_\eta a=a_\eta(-\infty)\chi_-+a_\eta(+\infty)\chi_+,
\end{equation}
where $a_\eta(\pm\infty)$ are defined by
\eqref{eq:functional-alpha-PC-SO}.

As $M_\infty(PC_p)=M_\infty(PC)=\{\pm\infty\}$ for every
$p\in(1,\infty)$, we infer that the restriction of the
homomorphism $\alpha_\eta$ to $[PC_p,SO_p]$ sends this algebra to
$PC_p|_{M_\infty(PC)}=PC|_{M_\infty(PC)}$ according to
\eqref{eq:functional-alpha-PC-SO}. Therefore, for all $\eta\in
M_\infty(SO)$, the homomorphisms $\gamma_\eta$ given by
\eqref{eq:functional-gamma-PC-SO} map $[PC_p,SO_p]$ into $PC_p$.

To have an idea how $\gamma_\eta$ acts, we give the following
example.
\begin{proof}[Example.]
Let $k\in\N$ and $c_k,d_k\in\C$. Put
\[
f_k(t):=c_k\chi_-(t)+d_k\chi_+(t), \quad
g_k(t):=\frac{t^2}{t^2+1}\exp\left(i\sqrt{\log(t^{2k}+1)}\right) \quad(t\in\R).
\]
The functions $f_k$ are piecewise constant with the only jumps at
the origin and infinity. So $f_k\in PC_p$ for every
$p\in(1,\infty)$. It is obvious that the functions $g_k$ belong to
$\widetilde{SO}^1\setminus C(\overline{\R})$ because $|g_k(t)|<1$
for all $t\in\R$,
\[
\lim_{t\to\infty}tg_k'(t)=
\lim_{t\to\infty}
\left(\frac{2}{t^2+1}+\frac{kt^{2k}}{t^{2k}+1}\cdot\frac{i}{\sqrt{\log(t^{2k}+1)}}\right)g_k(t)
=0
\]
and $g_k$ do not have limits as $t\to\pm\infty$.
Since $g_k$ are even, from Lemmas~\ref{le:tilde-SO1}--\ref{le:SO1}
it follows that $g_k\in SO^1\subset SO_p$ for all $p\in(1,\infty)$.
By Lemma~\ref{le:SO-partial-limits},
for every functional $\eta\in M_\infty(SO)$ there exists a sequence $\tau_n>1$,
$\tau_n\to+\infty$ such that
\[
\eta(g_k)=\lim_{n\to\infty}g_k(\tau_n)
\quad\mbox{for all}\quad k\in\N.
\]
Notice that $\eta(g_k)\in\T:=\{z\in\C:|z|=1\}$.

For every $m\in\N$, the function
\[
a=\sum_{k=1}^m f_kg_k
\]
belongs to the algebra $[PC_p,SO_p]$ and
\[
\begin{split}
(\alpha_\eta a)(-\infty) &=
\sum_{k=1}^m (\alpha_\eta f_k)(-\infty)\cdot(\alpha_\eta g_k)(-\infty)=
\sum_{k=1}^m c_k\eta(g_k),
\\
(\alpha_\eta a)(+\infty) &=
\sum_{k=1}^m(\alpha_\eta f_k)(+\infty)\cdot(\alpha_\eta g_k)(+\infty)=
\sum_{k=1}^m d_k\eta(g_k).
\end{split}
\]
Thus $\gamma_\eta a$ is the piecewise constant function
\[
\gamma_\eta a=\sum_{k=1}^m (c_k\chi_-+d_k\chi_+)\eta(g_k)\in PC_p.
\qedhere
\]
\end{proof}
Let $\cC^\pi$ denote the smallest closed subalgebra of the Calkin
algebra $\cB/\cK$ that contains all the cosets $aI+\cK$ with $a\in
C(\dot{\R})$ and $W^0(b)+\cK$ with $b\in SO_p$.
\begin{lemma}\label{le:maximal-ideal-Cpi}
The algebra $\cC^\pi$ is commutative. The maximal ideal space $M(\cC^\pi)$
of $\cC^\pi$ is homeomorphic to the set
\begin{equation}\label{eq:def-Omega}
\Omega:=
\big(\R\times M_\infty(SO)\big)
\cup
\big(\{\infty\}\times\R\big)
\cup
\big(\{\infty\}\times M_\infty(SO)\big).
\end{equation}
\end{lemma}
\begin{proof}
Lemma~\ref{le:compactness-commutator} implies that the algebra $\cC^\pi$
is commutative. Then in the same way as in \cite[Lemma~5.1]{BBK04-OT} (see also
\cite[Proposition~14.1]{RS90}) one can prove that its maximal ideal space
$M(\cC^\pi)$ is homeomorphic to $\Omega$.
\end{proof}
\section{Singular integral operators on connected subsets of \boldmath{$\R$}}
\label{section:SIO}
\subsection{Circular arcs}
In this subsection we will follow \cite[Section~9.1]{GK92} and \cite[Section~7.4]{BK97}.
Given two points $z_1,z_2\in\C$ and a number $s\in(1,\infty)$ one can define
the circular arc $\fA_s(z_1,z_2)$ between $z_1$ and $z_2$ by
\[
\fA_s(z_1,z_2):=\left\{ z\in\C\setminus\{0,1\}:\arg\frac{z-z_1}{z-z_2}\in
\frac{2\pi}{s}+2\pi\Z \right\}\cup\{z_1,z_2\}.
\]
If $z_1=z_2=:z$, then $\fA_s(z_1,z_2)$ is simply $\{z\}$.
The set $\fA_2(z_1,z_2)$ is the segment $[z_1,z_2]$,
if $s>2$ (resp. $1<s<2$), then $\fA_s(z_1,z_2)$ is the
circular arc at the points of which the segment $[z_1,z_2]$ is seen
at the angle $2\pi/s$ (resp. $2\pi-2\pi/s$) and which lies on the
right (resp. left) of the straight line passing first $z_1$ and
then $z_2$. The arc $\fA_s(z_1,z_2)$ can be analytically represented by
\[
z=z_1[1-f_s(\mu)]+z_2f_s(\mu)\quad(\mu\in[0,1],\quad z\in\fA_s(z_1,z_2)),
\]
where $f_s:[0,1]\to\C$ is defined by
\[
f_s(\mu):=\left\{
\begin{array}{lll}
\mu &\mbox{if}& s=2,\\[2mm]
\displaystyle
\frac{\sin(\pi\mu-2\pi\mu/s)}{\sin(\pi-2\pi/s)}e^{i(\pi-2\pi/s)(\mu-1)}&
\mbox{if}& s\in(1,2)\cup(2,\infty).
\end{array}\right.
\]
\subsection{The Gohberg-Krupnik theorem}
Let $-\infty\le\alpha<\beta\le+\infty$ and $(\alpha,\beta)\ne\R$.
For $\varphi\in L^1(\alpha,\beta)$, consider the Cauchy singular
integral operator $S_{(\alpha,\beta)}$ given by
\[
(S_{(\alpha,\beta)}\varphi)(t):= \frac{1}{\pi
i}\int_\alpha^\beta\frac{\varphi(\tau)}{\tau-t}\,d\tau,
\]
where the integral is understood in the principal value sense. Put
\[
P_{(\alpha,\beta)}:=(I+S_{(\alpha,\beta)})/2,\quad
Q_{(\alpha,\beta)}:=(I-S_{(\alpha,\beta)})/2.
\]
It is well known that these operators are bounded on
$L^p(\alpha,\beta)$ for $1<p<\infty$ (see e.g. \cite[Chap.~1,
Theorem~3.1]{GK92}). By $PC[\alpha,\beta]$ denote the collection of
all elements of $PC$ with a finite number of jumps restricted to $[\alpha,\beta]$.

For $a\in PC[\alpha,\beta]$, $p\in(1,\infty)$, and $\mu\in[0,1]$, put
\[
a_p(t,\mu):=\left\{
\begin{array}{lll}
[1-f_p(\mu)]+a(\alpha^+)f_p(\mu) &\mbox{if}& t=\alpha,
\\[2mm]
a(t^-)[1-f_p(\mu)]+a(t^+)f_p(\mu) &\mbox{if}& t\in(\alpha,\beta),
\\[2mm]
a(\beta^-)[1-f_p(\mu)]+f_p(\mu) &\mbox{if}&t=\beta.
\end{array}
\right.
\]
The range of this function is a closed continuous curve obtained
from the range of the function $a$ by adding the arcs
$\fA_p(a(t_k^-),a(t_k^+))$ for all jumps $t_k$ of $a$ and the arcs
$\fA_p(a(\beta^-),1)$ and $\fA_p(1,a(\alpha^+))$. It can be
oriented in the natural manner: on the intervals of continuity of
$a$, the motion along this curve agrees with the increment of $t$,
while the supplementary arcs $\fA_p(\cdot,\cdot)$ are oriented
from the point on the first position to the point on the second
position in the definition of the arc $\fA_p(\cdot,\cdot)$. If
$a_p(t,\mu)\ne 0$ for all $(t,\mu)\in[\alpha,\beta]\times[0,1]$,
then by $\operatorname{wind}a_p$ we denote the winding number of
this curve about the origin.

The Gohberg-Krupnik one-sided invertibility criteria for
singular integral operators with piecewise continuous coefficients
over the segment $[\alpha,\beta]$ read as follows
(see \cite[Chap.~9, Theorem~4.1]{GK92} and also
\cite[Chap.~IV, Theorems~5.1 and 6.1]{MP86}).
\begin{theorem}[Gohberg-Krupnik]
\label{th:GK} Let $-\infty\le\alpha<\beta\le+\infty$ and
$(\alpha,\beta)\ne\R$. Suppose $1<p<\infty$ and $c,d\in
PC[\alpha,\beta]$. The operator
$A=P_{(\alpha,\beta)}cI+Q_{(\alpha,\beta)}dI$ is at least
one-sided invertible on $L^p(\alpha,\beta)$ if and only if the
following conditions are satisfied for all $\mu\in[0,1]$:
\begin{equation}\label{eq:GK-condition}
\left\{\begin{array}{ll}
c(\alpha^+)f_p(\mu)+d(\alpha^+)[1-f_p(\mu)]\ne 0,
\\[2mm]
c(t^+)d(t^-)f_p(\mu)+c(t^-)d(t^+)[1-f_p(\mu)]\ne 0 &
(t\in(\alpha,\beta)),
\\[2mm]
c(\beta^-)[1-f_p(\mu)]+d(\beta^-)f_p(\mu)\ne 0.
\end{array}\right.
\end{equation}
If these conditions are fulfilled, then the operator $A$ is invertible,
invertible only from the left, invertible only from the right depending on
whether the number $\operatorname{wind}_p(c/d)$ is equal to zero, positive,
or negative, respectively.
\end{theorem}
\subsection{Spectra of the operators \boldmath{$P_{(\alpha,\beta)}$} and
\boldmath{$Q_{(\alpha,\beta)}$}}
For $p\in (1,\infty)$, put $q:=p/(p-1)$ and define the
lentiform domain $\fL_p$ by
\begin{equation}\label{eq:lens}
\fL_p:=\big\{z\in \fA_s(0,1)\ :\  \min\{p,q\}\le s\le\max\{p,q\}\big\}.
\end{equation}
By $\cB(L^p(\alpha,\beta))$ denote the Banach algebra of all bounded linear
operators on $L^p(\alpha,\beta)$.
The spectrum of an element $a$ in a unital Banach algebra $B$ will
be denoted by $\operatorname{sp}_B(a)$.
\begin{theorem}\label{th:spectrum-PQ}
Let $1<p<\infty$ and $\fL_p$ be the lentiform domain given by
\eqref{eq:lens}. If $-\infty\le\alpha<\beta\le+\infty$ and $(\alpha,\beta)\ne\R$,
then
\[
\operatorname{sp}_{\cB(L^p(\alpha,\beta))}(P_{(\alpha,\beta)})
=
\operatorname{sp}_{\cB(L^p(\alpha,\beta))}(Q_{(\alpha,\beta)})
=
\fL_p.
\]
\end{theorem}
\begin{proof}
Let $\lambda\in\C$. Then for the operator
\[
A_\lambda:=P_{(\alpha,\beta)}-\lambda I=P_{(\alpha,\beta)}(1-\lambda)I+Q_{(\alpha,\beta)}(-\lambda)I
\]
conditions \eqref{eq:GK-condition} with $c:=1-\lambda$ and
$d:=-\lambda$ have the form, for $\mu\in[0,1]$,
\[
c(\alpha^+)f_p(\mu)+d(\alpha^+)[1-f_p(\mu)] =
(1-\lambda)f_p(\mu)-\lambda[1-f_p(\mu)] =
f_p(\mu)-\lambda\ne 0;
\]
\[
\begin{split}
& c(t^+)d(t^-)f_p(\mu)+c(t^-)d(t^+)[1-f_p(\mu)]
\\
&=
(1-\lambda)(-\lambda)f_p(\mu)+(1-\lambda)(-\lambda)[1-f_p(\mu)]
=\lambda(\lambda-1)\ne 0
\end{split}
\]
whenever $t\in(\alpha,\beta)$; and
\[
c(\beta^-)[1-f_p(\mu)]+d(\beta^-)f_p(\mu)=
(1-\lambda)[1-f_p(\mu)]+(-\lambda)f_p(\mu)=
1-\lambda-f_p(\mu)\ne 0.
\]
These conditions are equivalent to
\begin{equation}\label{eq:spectrum-PQ-1}
\lambda\notin\fA_p(0,1),
\quad
1-\lambda\notin\fA_p(0,1).
\end{equation}
By Theorem~\ref{th:GK}, the operator $A_\lambda$ is one-sided invertible
if and only if \eqref{eq:spectrum-PQ-1} is fulfilled. Further, if
\eqref{eq:spectrum-PQ-1} holds, then
\[
\left(\frac{c}{d}\right)_{\!p}(t,\mu)=\left\{\begin{array}{lll}
\displaystyle \frac{\lambda-1}{\lambda}f_p(\mu)+[1-f_p(\mu)]
&\mbox{if}& t=\alpha,
\\[3mm]
\displaystyle
\frac{\lambda-1}{\lambda} &\mbox{if}& t\in(\alpha,\beta),
\\[3mm]
\displaystyle
\frac{\lambda-1}{\lambda}[1-f_p(\mu)]+f_p(\mu) &\mbox{if}& t=\beta.
\end{array}\right.
\]
So, the range of $(c/d)_p$ coincides with the closed curve
$\fA_p\left(1,\frac{\lambda-1}{\lambda}\right)\cup\fA_p\left(\frac{\lambda-1}{\lambda},1\right)$.
It is easy to see that
\[
\operatorname{wind}
\fA_p\left(1,\frac{\lambda-1}{\lambda}\right)\cup
\fA_p\left(\frac{\lambda-1}{\lambda},1\right)=0
\]
if and only if the origin lies outside the convex lentiform domain $\fM_p(\lambda)$
bounded by the arcs
$\fA_p\left(1,\frac{\lambda-1}{\lambda}\right)=\fA_q\left(\frac{\lambda-1}{\lambda},1\right)$,
where $1/p+1/q=1$, and $\fA_p\left(\frac{\lambda-1}{\lambda},1\right)$.

Let $I_p:=[\min\{p,q\},\max\{p,q\}]$. By Theorem~\ref{th:GK}, the
operator $A_\lambda$ is invertible on $L^p(\alpha,\beta)$ if and
only if \eqref{eq:spectrum-PQ-1} is fulfilled and
$0\notin\fM_p(\lambda)$. Then the spectrum of $P_{(\alpha,\beta)}$
is equal to $\sigma_1\cup\sigma_2\cup\sigma_2$, where
\[
\begin{split}
\sigma_1 &:=
\big\{\lambda\in\C:\lambda\in\fA_p(0,1)\big\}=\fA_p(0,1),
\\
\sigma_2 &:=
\big\{\lambda\in\C:1-\lambda\in\fA_p(0,1)\big\}=\fA_p(0,1),
\end{split}
\]
and
\[
\begin{split}
\sigma_3
&:=
\big\{\lambda\in\C\setminus(\sigma_1\cup\sigma_2)\ :\ 0\in\fM_p(\lambda)\big\}
\\
&=
\left\{\lambda\in\C\setminus\fA_p(0,1)\ :0\in\bigcup_{s\in I_p}
\fA_s\left(\frac{\lambda-1}{\lambda},1\right)\right\}
\\
&=
\left\{\lambda\in\C\setminus\fA_p(0,1)\ :\
\frac{\lambda-1}{\lambda}[1-f_s(\mu)]+f_s(\mu)=0
\mbox{ for some }s\in I_p\right\}
\\
&=
\big\{\lambda\in\C\setminus\fA_p(0,1)\ :\
1-\lambda\in\fA_s(0,1)
\mbox{ for some }s\in I_p\big\}
\\
&=
\big\{\lambda\in\C\setminus\fA_p(0,1)\ :\
1-\lambda\in\fL_p(0,1)\big\}
\\
&=
\big\{\lambda\in\C\setminus\fA_p(0,1)\ :\
\lambda\in\fL_p(0,1)\big\}
\\
&=
\fL_p\setminus\fA_p(0,1).
\end{split}
\]
Thus
\[
\operatorname{sp}_{\cB(L^p(\alpha,\beta))}(P_{(\alpha,\beta)})=
\sigma_1\cup\sigma_2\cup\sigma_3=\fA_p(0,1)\cup\fA_p(0,1)\cup(\fL_p\setminus\fA_p(0,1))=\fL_p.
\]
The proof of the equality $\operatorname{sp}_{\cB(L^p(\alpha,\beta))}(Q_{(\alpha,\beta)})=\fL_p$
is analogous.
\end{proof}
Let $\chi_E$ denote the characteristic function of a set $E\subset\R$.
\begin{corollary}\label{co:spectra-SIO}
Let $1<p<\infty$ and $\fL_p$ be the lentiform domain given by
\eqref{eq:lens}. Then
\begin{enumerate}
\item[{\rm (a)}]
$\operatorname{sp}_{\cB}\left(
\chi_{(-1,0)}P_\R\chi_{(-1,0)}I+
\chi_{(-\infty,-1)}Q_\R\chi_{(-\infty,-1)}I
\right)=\fL_p;$

\medskip
\item[{\rm (b)}]
$\operatorname{sp}_{\cB}\left(
\chi_{(1,\infty)}P_\R\chi_{(1,\infty)}I+
\chi_{(0,1)}Q_\R\chi_{(0,1)}I
\right)=\fL_p;$

\medskip
\item[{\rm (c)}]
$\operatorname{sp}_{\cB}\left(
\chi_+P_\R\chi_+I+
\chi_-Q_\R\chi_-I
\right)=\fL_p.$
\end{enumerate}
\end{corollary}
\begin{proof}
Let us prove equality (a). The operator
\[
\begin{split}
&
\chi_{(-1,0)}P_\R\chi_{(-1,0)}I+\chi_{(-\infty,-1)}Q_\R\chi_{(-\infty,-1)}I-\lambda I
\\
&=
\chi_{(-\infty,-1)}(Q_\R-\lambda I)\chi_{(-\infty,-1)}I+
\chi_{(-1,0)}(P_\R-\lambda I)\chi_{(-1,0)}I-
\lambda\chi_+I
\end{split}
\]
is represented in the direct sum
\[
L^p(-\infty,-1)\stackrel{\cdot}{+}L^p(-1,0)\stackrel{\cdot}{+}L^p(\R_+)=L^p(\R)
\]
by the matrix
\[
\left[\begin{array}{ccc}
Q_{(-\infty,-1)}-\lambda I & 0 & 0\\
0 & P_{(-1,0)}-\lambda I & 0\\
0 & 0 & -\lambda I
\end{array}\right].
\]
Hence, by Theorem~\ref{th:spectrum-PQ},
\[
\begin{split}
&
\operatorname{sp}_{\cB}\left(
\chi_{(-1,0)}P_\R\chi_{(-1,0)}I+
\chi_{(-\infty,-1)}Q_\R\chi_{(-\infty,-1)}I
\right)
\\
&=
\operatorname{sp}_{\cB(L^p(-\infty,-1))}(Q_{(-\infty,-1)})
\cup
\operatorname{sp}_{\cB(L^p(-1,0))}(P_{(-1,0)})
\cup
\operatorname{sp}_{\cB(L^p(\R_+))}(0)
\\
&=\fL_p\cup\fL_p\cup\{0\}=\fL_p.
\end{split}
\]
The proof of equalities (b)--(c) is
analogous.
\end{proof}
\subsection{Singular integral operator with piecewise
constant coefficients}\label{subsec:SIO-piecewise-constant}
Let $z_1,z_2,z_3$ be an ordered triple of points in the complex plane.
It is clear that
\[
\fA_p(z_1,z_2)\cup\fA_p(z_2,z_3)\cup\fA_p(z_3,z_1)=:\fC_p(z_1,z_2,z_3)
\]
is a closed curve in the complex plane (which degenerates to the point $z$
if $z=z_1=z_2=z_3)$. Each arc $\fA_p(z_1,z_2)$, $\fA_p(z_2,z_3)$, and
$\fA_p(z_3,z_1)$ can be naturally oriented starting from the point on the
first position and terminating at the point on the second position in
the definition of the arc $\fA_p(\cdot,\cdot)$. This orientation induces
the orientation of the curve $\fC_p(z_1,z_2,z_3)$. If $0\notin\fC_p(z_1,z_2,z_3)$,
then by $\operatorname{wind}\fC_p(z_1,z_2,z_3)$ we denote the winding number
of the curve $\fC_p(z_1,z_2,z_3)$ about the origin. By definition,
$\operatorname{wind}\fC_p(z,z,z)=0$.
\begin{lemma}\label{le:SIO-piecewise-constant}
Let $1<p<\infty$. Suppose $a_-,a_+,b_-,b_+\in\C$. The singular integral
operator
\[
P_{(-1,1)}(a_-\chi_{(-1,0)}+b_-\chi_{(0,1)})I+
Q_{(-1,1)}(a_+\chi_{(-1,0)}+b_+\chi_{(0,1)})I
\]
is invertible on the space $L^p(-1,1)$ if and only if
\begin{equation}\label{eq:SIO-piecewise-constant-1}
0\notin\{a_-,a_+,b_-,b_+\},
\quad
\operatorname{wind}\fC_p\left(1,\frac{a_-}{a_+},\frac{b_-}{b_+}\right)=0.
\end{equation}
\end{lemma}
This lemma follows from Theorem~\ref{th:GK}. Its proof is similar
to the proof of Theorem~\ref{th:spectrum-PQ} and it is omitted
here.
\section{Algebraization}\label{section:algebraization}
\subsection{Stability of the finite section method}
Let $\cE$ be the set formed by all the sequences $(A_\tau)$
(depending on a parameter $\tau\in\R_+$) of operators
$A_\tau\in\cB$ such that
\[
\sup\limits_{\tau\in\R_+}\|A_\tau\|_\cB < \infty.
\]
One says that the sequence $(A_\tau)\in\cE$ is \textit{stable} if
there exists a positive constant $\tau_0$ such that
the operator $A_\tau\in\cB$ is invertible for any $\tau>\tau_0$
and
\[
\sup_{\tau>\tau_0} \|A_\tau^{-1}\|_\cB < \infty.
\]
The following result is well known (see e.g. \cite[Proposition~1.1]{HRS95}
and also \cite[Proposition~7.3]{BS06}).
\begin{theorem}[Polski]
\label{th:Polski}
The finite section method applies to an operator $A\in\cB$
if and only if $A$ is invertible and the sequence $(P_\tau A P_\tau+ Q_\tau)$
is stable.
\end{theorem}
\subsection{Algebra \boldmath{$\cE$} and its ideal \boldmath{$\cG$}}
\begin{lemma}\label{le:algebra-E}
The set $\cE$ with the operations
\[
(A_\tau) + (B_\tau):=(A_\tau + B_\tau), \quad
(A_\tau)(B_\tau):=(A_\tau B_\tau), \quad \lambda(A_\tau):=(\lambda
A_\tau) \quad (\lambda\in\C),
\]
the identity element $(I)$, and the norm
\[
\|(A_\tau)\|_\cE:=\sup_{\tau\in\R_+}\|A_\tau\|_\cB
\]
forms a unital Banach algebra.
\end{lemma}
This statement is proved as \cite[Proposition~1.13]{HRS01}.

Note that the constant sequences $(A)$ are included in $\cE$ for
any $A\in\cB$.
\begin{definition}\label{def:alg-A}
By $\cA$ denote the smallest Banach subalgebra of
$\cE$ that contains all constant sequences $(aI)$ with $a\in PC$
and $(W^0(b))$ with $b\in[PC_p,SO_p]$ and the sequence $(P_\tau)$.
\end{definition}
We will be interested in conditions for the stability of any
sequence in $\cA$. Clearly, $\cA$ contains the sequences \eqref{eq:sequence-FSM}.

Let $\cG$ be the set of all sequences $(A_\tau)\in\cE$ satisfying
\[
\lim_{\tau\to\infty}\|A_\tau\|_\cB = 0.
\]
\begin{lemma}\label{le:ideal-G-of-E}
The set $\cG$ is a closed two-sided ideal of the algebra $\cE$.
\end{lemma}
The proof is similar to the proof of
\cite[Proposition~1.14]{HRS01}.
\begin{theorem}[Kozak]
\label{th:Kozak} Let $(A_\tau)\in\cE$. The sequence $(A_\tau)$ is
stable if and only if the coset $(A_\tau)+\cG$ is invertible in
the quotient algebra $\cE/\cG$.
\end{theorem}
This result is well known (see e.g. \cite[Proposition~7.3]{BS06},
\cite[Proposition~1.2]{HRS95}, \cite[Theorem~1.15]{HRS01}).
\section{Essentialization}\label{section:essentialization}
\subsection{Algebra \boldmath{$\cF$} and its ideal \boldmath{$\cG$}}
A sequence of operators on a Banach space is said to converge
\textit{*-strongly} if it converges strongly and the sequence of
the adjoint operators converges strongly on the dual space.

Let $\tau\in\R_+$. By $V_\tau$ denote the additive shift operator given
by
\[
(V_\tau f)(x):=f(x-\tau)\quad (x\in\R).
\]
It is clear that $V_\tau$ is bounded and invertible on $L^p(\R)$ and
$V_\tau^{-1}=V_{-\tau}$. Moreover, $\|V_\tau\|_\cB=1$ for all $\tau\in\R$.

Let $\cF$ denote the set of all sequences $\bA:=(A_\tau)\in\cE$ such
that the sequences $(A_\tau)$, $(V_{-\tau}A_\tau V_\tau)$, and
$(V_\tau A_\tau V_{-\tau})$ are *-strongly convergent as $\tau\to
+\infty$.
\begin{lemma}\label{le:alg-F}
\begin{enumerate}
\item[{\rm (a)}]
The set $\cF$ is a closed unital subalgebra of the algebra $\cE$.

\item[{\rm (b)}]
Let $i\in\{-1,0,1\}$. The mappings $\fW_i:\cF\to\cB$ given by
\[
\begin{split}
\fW_{-1}(\bA) &:=\operatornamewithlimits{s-\lim}_{\tau\to+\infty}V_\tau A_\tau V_{-\tau},
\\
\fW_0(\bA) &:=\operatornamewithlimits{s-lim}_{\tau\to+\infty}A_\tau,
\\
\fW_1(\bA) &:=\operatornamewithlimits{s-\lim}_{\tau\to+\infty}V_{-\tau} A_\tau V_\tau
\end{split}
\]
for $\bA=(A_\tau)\in\cF$ are bounded unital homomorphisms with the norms
\[
\|\fW_{-1}\|=\|\fW_0\|=\|\fW_1\|=1.
\]

\item[{\rm (c)}]
The set $\cG$ is a closed two-sided ideal of the algebra $\cF$.

\item[{\rm (d)}]
The ideal $\cG$ lies in the kernel of each homomorphism $\fW_i$ for $i\in\{-1,0,1\}$.

\item[{\rm (e)}]
The algebra $\cF$ is inverse closed in the algebra $\cE$ and the algebra
$\cF/\cG$ is inverse closed in the algebra $\cE/\cG$.
\end{enumerate}
\end{lemma}
The proof follows the proof of \cite[Proposition~4.1]{RSS09}.
Notice that the algebra $\cF$ in the present paper is larger than the algebra
considered in  \cite{RSS09} and also denoted there by $\cF$ (see also
Remark~\ref{rem:algebras-F}).
\subsection{The algebra \boldmath{$\cA$} is contained in the algebra \boldmath{$\cF$}}
\begin{proposition}\label{pr:A-in-F}
\begin{enumerate}
\item[{\rm (a)}]
If $\bP=(P_\tau)$, then $\bP\in\cF$ and
\begin{equation}\label{eq:A-in-F-1}
\fW_{-1}(\bP)=\chi_+I, \quad
\fW_0(\bP)=I,\quad
\fW_1(\bP)=\chi_-I.
\end{equation}

\item[{\rm (b)}]
If $\bA=(aI)$ with $a\in PC$, then $\bA\in\cF$ and
\begin{equation}\label{eq:A-in-F-2}
\fW_{-1}(\bA)=a(-\infty)I, \quad
\fW_0(\bA)=aI,\quad
\fW_1(\bA)=a(+\infty)I.
\end{equation}

\item[{\rm (c)}]
If $\bB=(W^0(b))$ with $b\in [PC_p,SO_p]$, then $\bB\in\cF$ and
\begin{equation}\label{eq:A-in-F-3}
\fW_{-1}(\bB)=W^0(b), \quad
\fW_0(\bB)=W^0(b),\quad
\fW_1(\bB)=W^0(b).
\end{equation}
\end{enumerate}
\end{proposition}
\begin{proof}
The proof of equalities \eqref{eq:A-in-F-1} and
\eqref{eq:A-in-F-2} is straightforward, equalities in
\eqref{eq:A-in-F-3} are trivial because the convolution operator
$W^0(b)$ is translation-invariant, that is,
$V_{\pm\tau}W^0(b)V_{\mp\tau}=W^0(b)$ for every $b\in\cM_p$, in
particular, for every $b\in[PC_p,SO_p]$. To finish the proof, it
remains to note that
\[
P_\tau^*=P_\tau\in\cB(L^q(\R)),\quad
[aI]^*=\overline{a}I\in\cB(L^q(\R)),\quad
[W^0(b)]^*=W^0(\overline{b})\in\cB(L^q(\R)),
\]
where $1/p+1/q=1$. Therefore the existence of the strong limits
for the adjoint operators in the definition of the algebra $\cF$
can be obtained in the same way as the existence of the strong
limits in \eqref{eq:A-in-F-1}--\eqref{eq:A-in-F-3}.
\end{proof}
\begin{corollary}\label{co:A-in-F}
The algebra $\cA$ is a closed unital subalgebra of the algebra $\cF$.
\end{corollary}
\subsection{Ideals \boldmath{$\cJ_{-1}$},
\boldmath{$\cJ_0$}, \boldmath{$\cJ_1$}, and \boldmath{$\cJ$} of the algebra
\boldmath{$\cF$}}
Recall that $\cK$ denotes the ideal of the compact operators on $L^p(\R)$.
Put
\[
\begin{split}
\cJ_{-1} &:=\big\{(V_{-\tau}KV_\tau)+(G_\tau)\ :\ K\in\cK,\ (G_\tau)\in\cG\big\},
\\
\cJ_0 &:=\big\{(K)+(G_\tau)\ :\ K\in\cK,\ (G_\tau)\in\cG\big\},
\\
\cJ_1 &:=\big\{(V_\tau KV_{-\tau})+(G_\tau)\ :\ K\in\cK,\ (G_\tau)\in\cG\big\},
\end{split}
\]
and
\[
\cJ:=\big\{ (V_\tau K_1V_{-\tau})+(K_0)+(V_{-\tau} K_{-1}
V_\tau)+(G_\tau)\ :\ K_{-1},K_0,K_1\in\cK,\ (G_\tau)\in\cG \big\}.
\]
\begin{lemma}\label{le:ideals-J}
The sets $\cJ_i$ with $i\in\{-1,0,1\}$ and $\cJ$ are closed two-sided ideals
of the algebra $\cF$.
\end{lemma}
\begin{proof}
First we note that $(V_{\pm\tau})$ converges weakly to zero as
$\tau\to+\infty$. Let us show that $\cJ_{-1}\subset\cF$. Let
$\bJ=(J_\tau)=(V_{-\tau}KV_\tau)+(G_\tau)$ for some $K\in\cK$ and
$(G_\tau)\in\cG$. Since $K$ is compact and $V_{\pm\tau}^2=V_{\pm
2\tau}$, we obtain that $(J_\tau)$ is *-strongly convergent to
zero, $(V_{-\tau}J_\tau V_\tau)$ is *-strongly convergent to zero,
and $(V_\tau J_\tau V_{-\tau})$ is *-strongly convergent to $K$.
Thus $\cJ_{-1}\subset\cF$.

Let $\bA=(A_\tau)\in\cF$. Then
\[
\begin{split}
\bA\bJ &= (A_\tau V_{-\tau} KV_\tau)+(A_\tau G_\tau)
\\
&=(V_{-\tau} V_\tau A_\tau V_{-\tau}KV_\tau)+(A_\tau G_\tau)
\\
&= (V_{-\tau}\fW_{-1}(\bA)KV_\tau)+
(V_{-\tau}[V_\tau A_\tau V_{-\tau}K-\fW_{-1}(\bA)K]V_\tau+A_\tau G_\tau)
\\
&=:(V_{-\tau}\fW_{-1}(\bA)KV_\tau)+(G_\tau').
\end{split}
\]
Since the sequence $(V_\tau A_\tau V_{-\tau}-\fW_{-1}(\bA))$
converges strongly to zero and $K\in\cK$, the sequence  $(V_\tau
A_\tau V_{-\tau}K-\fW_{-1}(\bA)K)$ converges uniformly to zero.
Thus
\[
(G_\tau'):=(V_{-\tau}[V_\tau A_\tau
V_{-\tau}K-\fW_{-1}(\bA)K]V_{-\tau}+A_\tau G_\tau)\in\cG
\]
and $\bA\bJ\in\cJ_{-1}$. Analogously one can show that
$\bJ\bA\in\cJ_{-1}$. This means that $\cJ_{-1}$ is a two-sided
ideal of $\cF$.

Let us show that $\cJ_{-1}$ is closed. Suppose $(\bJ_k)_{k\in\N}$
is a Cauchy sequence in $\cJ_{-1}$. Since $\cF$ is closed in $\cE$
by Lemma~\ref{le:alg-F}(a), the sequence $(\bJ_k)_{k\in\N}$ is
convergent to some $\bJ=(J_\tau)\in\cF$. We claim that
$\bJ\in\cJ_{-1}$. Dy definition of $\cJ_{-1}$, there exist
sequences $(K^{(j)})_{j\in\N}\subset\cK$ and
$(G_\tau^{(j)})_{j\in\N}\subset\cG$ such that
\[
\bJ_j=(J_\tau^{(j)})=(V_{-\tau}K^{(j)}V_\tau)+(G_\tau^{(j)}).
\]
From Lemma~\ref{le:alg-F}(b) it follows that for all
$j,k\in\N$,
\[
\|K^{(j)}-K^{(k)}\|_\cB
=
\|\fW_{-1}(\bJ_j)-\fW_{-1}(\bJ_k)\|_\cB
\le
\|\bJ_j-\bJ_k\|_\cE.
\]
Then $(K^{(j)})_{j\in\N}$ is a Cauchy sequence in $\cB$. Let
$K\in\cK$ be its limit. Put
\[
\bG:=(G_\tau):=(J_\tau)-(V_{-\tau}KV_\tau).
\]
Taking into account that $\|V_{\pm\tau}\|_\cB=1$, we obtain
\[
\begin{split}
\|(G_\tau)-(G_\tau^{(j)})\|_\cE
&\le
\|(J_\tau)-(J_\tau^{(j)})\|_\cE+
\|(V_{-\tau}KV_\tau)-(V_{-\tau}K^{(j)}V_\tau)\|_\cE
\\
&\le
\|\bJ-\bJ_j\|_\cE+\|K-K^{(j)}\|_\cB.
\end{split}
\]
Hence $\bG$ is the limit of $(\bG_j)_{j\in\N}$ with
$\bG_j:=(G_\tau^{(j)})\in\cG$ in the norm of $\cE$. By
Lemma~\ref{le:alg-F}(c), $\bG\in\cG$. Thus $\bJ\in\cJ_{-1}$ and
$\cJ_{-1}$ is closed.

Analogously it can be shown that $\cJ_0$, $\cJ_{-1}$, and $\cJ$ are closed
two-sided ideals of the algebra $\cF$.
\end{proof}
\subsection{Lifting}
The main result of this section is the following lifting theorem adapted for
our purposes from \cite[Theorem 1.8]{HRS95}.
\begin{theorem}\label{th:lifting-our}
Let $\bA=(A_\tau)\in\cF$. The coset $\bA+\cG$ is invertible in the
quotient algebra $\cF/\cG$ if and only if the operators
$\fW_{-1}(\bA)$, $\fW_0(\bA)$, and $\fW_1(\bA)$ are invertible in
$\cB$ and the coset $\bA+\cJ$ is invertible in the quotient
algebra $\cF/\cJ$.
\end{theorem}
\begin{proof}
The proof is analogous to the proof of the abstract lifting theorem
\cite[Theorem~1.8]{HRS95}, although we consider the invertibility of
$\fW_i(\bA)$ in $\cB$, while there it is considered in the images of the homomorphisms
$\fW_i$.

\medskip
\textit{Necessity.}
If $\bA+\cG$ is invertible in $\cF/\cG$, then there exist sequences
$\bB\in\cF$ and $\bG_1,\bG_2\in\cG$ such that
\begin{equation}\label{eq:lift-1}
\bA\bB=\bI+\bG_1,
\quad
\bB\bA=\bI+\bG_2.
\end{equation}
Let $i\in\{-1,0,1\}$. From Lemma~\ref{le:alg-F}(b),(d) we know
that $\fW_i:\cF\to\cB$ is a bounded unital homomorphism and
$\fW_i(\bG_1)=\fW_i(\bG_2)=0$. Applying this homomorphism to the
equalities in \eqref{eq:lift-1}, we conclude that the operator
$\fW_i(\bA)$ is invertible in $\cB$ and its inverse is
$\fW_i(\bB)$. From the definition of the ideal $\cJ$ it follows
that $\bG_1,\bG_2\in\cJ$. Thus the equalities \eqref{eq:lift-1}
imply also that $\bA+\cJ$ is invertible in $\cF/\cJ$. The
necessity part is proved.

\medskip
\textit{Sufficiency.}
Assume that the operators $\fW_{-1}(\bA)$, $\fW_0(\bA)$, $\fW_1(\bA)$
are invertible in $\cB$ and the coset $\bA+\cJ$ is invertible
in the quotient algebra $\cF/\cJ$. Then there exist sequences $\bB\in\cF$
and $\bJ\in\cJ$ such that $\bB\bA=\bI+\bJ$. From the definition of $\cJ$
we conclude that there exist elements $\bJ_i\in\cJ_i$ such that
$\bJ=\bJ_{-1}+\bJ_0+\bJ_1$.
Taking into account that, by Lemma~\ref{le:alg-F}(d),
\begin{eqnarray}
&&
\fW_{-1}[(V_{-\tau}KV_\tau)+(G_\tau)]=K,
\nonumber
\\
&&
\fW_0[(K)+(G_\tau)]=K,
\label{eq:lift-2}
\\
&&
\fW_1[(V_\tau KV_{-\tau})+(G_\tau)]=K
\nonumber
\end{eqnarray}
for every $K\in\cK$ and every $(G_\tau)\in\cG$, we have that $\fW_i$ maps
the ideal $\bJ_i$ onto the the ideal $\cK$ for $i\in\{-1,0,1\}$.
Since $\fW_i(\bJ_i)\in\cK$, we obviously have $\fW_i(\bJ_i)\fW_i(\bA)^{-1}\in\cK$.
Consequently
for every $i\in\{-1,0,1\}$, there exists $\bJ_i'\in\cJ_i$ such that
\begin{equation}\label{eq:lift-3}
\fW_i(\bJ_i')=\fW_i(\bJ_i)\fW_i(\bA)^{-1}.
\end{equation}
Put $\bB':=\bB-\bJ_{-1}'-\bJ_0'-\bJ_1'$. Then $\bB'+\cJ=\bB+\cJ$ and
\begin{eqnarray}
\bB'\bA
&=&
\bI+\bJ-\bJ_{-1}'\bA-\bJ_0'\bA-\bJ_1'\bA
\nonumber
\\
&=&
\bI+(\bJ_{-1}-\bJ_{-1}'\bA)+(\bJ_0-\bJ_0'\bA)+(\bJ_1-\bJ_1'\bA).
\label{eq:lift-4}
\end{eqnarray}
From \eqref{eq:lift-3} it follows that
$\fW_i(\bJ_i)=\fW_i(\bJ_i')\fW_i(\bA)=\fW_i(\bJ_i'\bA)$.
Hence
\[
\fW_i(\bJ_i-\bJ_i'\bA)=0.
\]
From this observation and \eqref{eq:lift-2}
we conclude that $\bJ_i-\bJ_i'\bA\in\cG$ for $i\in\{-1,0,1\}$.
Thus \eqref{eq:lift-4} can be written as
\[
\bB'\bA=\bI+\bG
\]
with $\bG=(\bJ_{-1}-\bJ_{-1}'\bA)+(\bJ_0-\bJ_0'\bA)+(\bJ_1-\bJ_1'\bA)\in\cG$.
This means that $\bA+\cG$ is
left-invertible in $\cF/\cG$. Analogously it can be shown that $\bA+\cG$
is right-invertible in $\cF/\cG$.
\end{proof}
\begin{remark}
A detailed discussion of various versions of abstract lifting
theorems and their applications in the setting of Banach algebras
and $C^*$-algebras is contained in \cite[Chap.~6]{RSS08}. In
particular, the above theorem follows from the Banach algebra
inverse closed lifting theorem (see \cite[Theorem~6.2.8]{RSS08}).
\end{remark}
\section{Localization}\label{section:localization}
\subsection{Allan-Douglas local principle}
\label{section:Allan-Douglas} Recall that the \textit{center} of
an algebra $A$ consists of all elements $a\in A$ such that $ab=ba$
for all $b\in A$. Let $A$ be a Banach algebra with identity. By a \textit{central
subalgebra} $C$ of $A$ one means a subalgebra of the center of
$A$.
\begin{theorem}[{\cite[Theorem~1.35(a)]{BS06} or \cite[Theorem~1.5]{HRS95}}]
\label{th:AllanDouglas} Let $A$ be a Banach algebra with identity
$e$ and let $C$ be closed central subalgebra of $A$ containing
$e$. Let $M(C)$ be the maximal ideal space of $C$, and for
$\omega\in M(C)$, let $I_\omega$ refer to the smallest closed
two-sided ideal of $A$ containing the ideal $\omega$. Then an
element $a\in A$ is invertible in $A$ if and only if $a+I_\omega$
is invertible in the quotient algebra $A/I_\omega$ for all
$\omega\in M(C)$.
\end{theorem}
\subsection{Sequences of local type}
We say that a sequence $(A_\tau)\in\cF$ is of \textit{local type}
if
\[
(A_\tau)(fI)-(fI)(A_\tau)\in\cJ, \quad
(A_\tau)(W^0(g))-(W^0(g))(A_\tau)\in\cJ
\]
for all $f\in C(\dot{\R})$ and all $g\in SO_p$. Let $\cL$ denote
the set of all sequences of local type.
\begin{lemma}\label{le:alg-L}
\begin{enumerate}
\item[{\rm (a)}]
The set $\cL$ is a closed unital subalgebra of the algebra $\cF$.
\item[{\rm (b)}]
The set $\cJ$ is a closed two-sided ideal of the algebra $\cL$.
\item[{\rm (c)}]
The algebra $\cL/\cJ$ is inverse closed in the algebra $\cF/\cJ$.
\item[{\rm (d)}]
The algebra $\cL$ is inverse closed in the algebra $\cF$.
\end{enumerate}
\end{lemma}
\begin{proof}
(a) By definition, $\cL\subset\cF$. It is clear that the sequence
$(I)$ is the identity of $\cL$. Hence $\cL$ is a unital subalgebra
of $\cF$. Let us show that $\cL$ is closed. Suppose
$(\bA_j)_{j\in\N}$ with $\bA_j=(A_\tau^{(j)})\in\cL$ is a Cauchy
sequence in $\cL$. Let $\bA=(A_\tau)\in\cF$ be its limit. We show
that $\bA\in\cL$. By definition of $\cL$,
\begin{equation}\label{eq:alg-L-1}
\bA_j(fI)-(fI)\bA_j\in\cJ,\quad
\bA_j(W^0(g))-(W^0(g))\bA_j\in\cJ
\end{equation}
for all $j\in\N$, $f\in C(\dot{\R})$, and $g\in SO_p$. Passing to
the limit in \eqref{eq:alg-L-1} as $j\to\infty$ and taking into
account that $\cJ$ is closed in $\cF$, we conclude that
$\bA(fI)-(fI)\bA\in\cJ$ and $\bA(W^0(g))-(W^0(g))\bA\in\cJ$ for
all $f\in C(\dot{\R})$ and $g\in SO_p$. Thus $\bA\in\cL$. Part (a)
is proved.

\medskip
(b) Since $\cJ\subset\cL\subset\cF$ and $\cJ$ is a closed
two-sided ideal of the algebra $\cF$, we also observe that $\cJ$
is a closed two-sided ideal of $\cL$. Part (b) is proved.

\medskip
(c) Let $\bA=(A_\tau)\in\cL$ and $\bA+\cJ$ be invertible in
$\cF/\cJ$. Then there exist sequences $\bB=(B_\tau)\in\cF$ and
$\bJ_1=(J_\tau^{(1)})$, $\bJ_2=(J_\tau^{(2)})$ belonging to $\cJ$
and such that
\[
I-A_\tau B_\tau=J_\tau^{(1)},\quad I-B_\tau A_\tau=J_\tau^{(2)}.
\]
Let $C$ be one of the operators $fI$ with $f\in C(\dot{\R})$ or
$W^0(g)$ with $g\in SO_p$. Then
\[
C
=
CA_\tau B_\tau+CJ_\tau^{(1)}
=
A_\tau CB_\tau+[CA_\tau-A_\tau C]B_\tau+CJ_\tau^{(1)}
\]
and
\[
\begin{split}
B_\tau C-CB_\tau &=
B_\tau A_\tau CB_\tau+B_\tau[CA_\tau-A_\tau C]B_\tau+B_\tau CJ_\tau^{(1)}-CB_\tau
\\
&=[CB_\tau-J_\tau^{(2)}CB_\tau]+B_\tau[CA_\tau-A_\tau C]B_\tau+B_\tau CJ_\tau^{(1)}-CB_\tau
\\
&=B_\tau[CA_\tau-A_\tau C]B_\tau+[B_\tau CJ_\tau^{(1)}-J_\tau^{(2)}CB_\tau].
\end{split}
\]
If we put $\bC=(C)$, then $\bB\bC-\bC\bB=\bB[\bC\bA-\bA\bC]\bB+\bJ$,
where $\bC\bA-\bA\bC\in\cJ$ and $\bJ:=\bB\bC\bJ_1-\bJ_2\bC\bB\in\cJ$.
Thus $\bB\bC-\bC\bB\in\cJ$ and $\bB+\cJ\in\cL/\cJ$. Part (c) is
proved.

\medskip
(d) Part (d) follows from part (c).
\end{proof}
\begin{remark}
The algebra $\cL$ is larger than the algebra of sequences of local
type considered in \cite{RSS09} and denoted there by $\cF_0$ (see
also Remark~\ref{rem:algebras-F}).
\end{remark}
\subsection{Sequences in \boldmath{$\cA$} are of local type}
\begin{theorem}\label{th:A-local-type}
The algebra $\cA$ is a closed unital subalgebra of $\cL$.
\end{theorem}
\begin{proof}
The idea of the proof is borrowed from
\cite[Proposition~4.11]{RSS09}.

From Corollary~\ref{co:A-in-F} we know that $\cA\subset\cF$.
Suppose $f\in C(\dot{\R})$ and $g\in SO_p$. It is sufficient to
prove that $(aI)$ with $a\in PC$, $(W^0(b))$ with
$b\in[PC_p,SO_p]$, and $(P_\tau)$ commute with $(fI)$ and
$(W^0(g))$ modulo the ideal $\cJ$.

Obviously $aI$ commutes with $fI$. By
Lemma~\ref{le:compactness-commutator}, $aW^0(g)-W^0(g)aI$ is
compact. Therefore $(aI)(fI)-(fI)(aI)\in\cJ$ and
$(aI)(W^0(g))-(W^0(g))(aI)\in\cJ$. Hence $(aI)\in\cL$.

It is clear that $W^0(b)$ commutes with $W^0(g)$. Applying
Lemma~\ref{le:compactness-commutator} once again, we see that
$fW^0(b)-W^0(b)fI$ is compact. Then
$(W^0(b))(fI)-(fI)(W^0(b))\in\cJ$ and
$(W^0(b))(W^0(g))-(W^0(g))(W^0(b))\in\cJ$. Thus $(W^0(b))\in\cL$.

Since $P_\tau$ commutes with $fI$, it is trivial that
$(P_\tau)(fI)-(fI)(P_\tau)\in\cJ$. Consider
\begin{eqnarray}
\hspace{-0.5cm}
(P_\tau W^0(g)-W^0(g)P_\tau) &=&
(P_\tau W^0(g) Q_\tau- Q_\tau W^0(g)P_\tau)
\nonumber
\\
\hspace{-0.5cm}
&=& (P_\tau\chi_+ W^0(g)\chi_+ Q_\tau- Q_\tau \chi_+W^0(g)\chi_+ P_\tau)
\nonumber
\\
\hspace{-0.5cm}
&&+ (P_\tau\chi_+ W^0(g)\chi_- Q_\tau- Q_\tau \chi_+W^0(g)\chi_- P_\tau)
\nonumber
\\
\hspace{-0.5cm}
&&+ (P_\tau\chi_- W^0(g)\chi_+ Q_\tau- Q_\tau \chi_-W^0(g)\chi_+ P_\tau)
\nonumber
\\
\hspace{-0.5cm}
&&+ (P_\tau\chi_- W^0(g)\chi_- Q_\tau- Q_\tau \chi_-W^0(g)\chi_- P_\tau).
\label{eq:A-local-type-1}
\end{eqnarray}
By Lemma~\ref{le:compactness-Hankel}, the operators
$\chi_+W^0(g)\chi_-I$ and $\chi_-W^0(g)\chi_+I$ are compact. Since
$(Q_\tau)$ converges *-strongly to zero, we conclude that the
sequences of the second and the third line on the right-hand side
of \eqref{eq:A-local-type-1} belong to $\cG$ and thus to $\cJ$. It
remains to show that
\begin{eqnarray}
&&(P_\tau\chi_+ W^0(g)\chi_+ Q_\tau- Q_\tau \chi_+W^0(g)\chi_+ P_\tau) \in\cJ,
\label{eq:A-local-type-2}
\\
&&(P_\tau\chi_- W^0(g)\chi_- Q_\tau- Q_\tau \chi_-W^0(g)\chi_- P_\tau) \in\cJ.
\label{eq:A-local-type-3}
\end{eqnarray}
It is easy to see that
\[
\begin{array}{ll}
V_{-\tau}P_\tau\chi_+I=\chi_{(-\tau,0)}V_{-\tau}, &
\chi_+Q_\tau V_\tau=V_\tau\chi_+I,\\[2mm]
V_{-\tau}Q_\tau\chi_+I=\chi_+ V_{-\tau}, &
\chi_+P_\tau V_\tau=V_\tau\chi_{(-\tau,0)}I,
\end{array}
\]
where $\chi_{(-\tau,0)}$ denotes the characteristic function of
the interval $(-\tau,0)$, and
\[
V_\tau V_{-\tau}=I,
\quad
V_{-\tau}W^0(g)V_\tau=W^0(g).
\]
Then
\begin{eqnarray*}
&&
(P_\tau\chi_+ W^0(g)\chi_+ Q_\tau- Q_\tau \chi_+W^0(g)\chi_+ P_\tau)
\\
&&=
(V_\tau V_{-\tau}[P_\tau\chi_+ W^0(g)\chi_+ Q_\tau- Q_\tau \chi_+W^0(g)\chi_+ P_\tau]V_\tau V_{-\tau})
\\
&&=
(V_\tau[\chi_{(-\tau,0)}V_{-\tau}W^0(g)V_\tau\chi_+I-\chi_+V_{-\tau}W^0(g)V_\tau\chi_{(-\tau,0)}I]V_{-\tau})
\end{eqnarray*}
\begin{eqnarray}
&&= (V_\tau[\chi_-W^0(g)\chi_+I-\chi_+W^0(g)\chi_-I]V_{-\tau})
\nonumber
\\
&&\quad+ (V_\tau[\chi_{(-\tau,0)}I-\chi_-I]\chi_-W^0(g)\chi_+V_{-\tau})
\nonumber
\\
&&\quad-
(V_\tau\chi_+W^0(g)\chi_-[\chi_{(-\tau,0)}I-\chi_-I]V_{-\tau}).
\label{eq:A-local-type-4}
\end{eqnarray}
The operators $\chi_-W^0(g)\chi_+I$ and $\chi_+W^0(g)\chi_-I$ are
compact by Lemma~\ref{le:compactness-Hankel}. It is easy to see
that $\chi_{(-\tau,0)}I$ is *-strongly convergent to $\chi_-I$.
Thus the sequence on the right hand side of
\eqref{eq:A-local-type-4} has the form $(V_\tau
K_1V_{-\tau})+(G_\tau^{(1)})$, where $K_1\in\cK$ and
$(G_\tau^{(1)})\in\cG$. This proves \eqref{eq:A-local-type-2}.
Analogously, it can be shown that
\[
(P_\tau\chi_- W^0(g)\chi_- Q_\tau- Q_\tau \chi_-W^0(g)\chi_-
P_\tau) =(V_{-\tau} K_{-1}V_\tau)+(G_\tau^{(2)}),
\]
where $K_{-1}\in\cK$ and $(G_\tau^{(2)})\in\cG$. This proves
\eqref{eq:A-local-type-3} and finishes the proof of the theorem.
\end{proof}
\subsection{Central subalgebra of the algebra \boldmath{$\cL/\cJ$} and its
maximal ideal space}
Let $\cC$ be the smallest closed subalgebra of $\cL$ that contains
all sequences $(fI)$ with $f\in C(\dot{\R})$ and $(W^0(g))$ with
$g\in SO_p$. From the proof of Theorem~\ref{th:A-local-type} it
follows that $\cC$ is not trivial. Put
\[
\cC^\cJ:=(\cC+\cJ)/\cJ,
\quad
\cL^\cJ:=\cL/\cJ.
\]
\begin{lemma}
The set $\cC^\cJ$ is a closed central subalgebra of the algebra $\cL^\cJ$.
\end{lemma}
\begin{proof}
This fact follows immediately from the definition of the algebras $\cL$,
$\cL^\cJ$, and $\cC^\cJ$.
\end{proof}
\begin{theorem}
The maximal ideal space $M(\cC^\cJ)$ of the commutative Banach algebra
$\cC^\cJ$ is homeomorphic to the set $\Omega$ given by \eqref{eq:def-Omega}.
\end{theorem}
\begin{proof}
Since $(V_{\pm\tau})$ converge weakly to zero as
$\tau\to+\infty$, it is easy to see that $\fW_0(\bJ)\in\cK$ for
every $\bJ\in\cJ$. Hence
\[
\Phi:\cF/\cJ\to\cB/\cK,\quad\bA+\cJ\mapsto\fW_0(\bA)+\cK
\]
is a well defined homomorphism. Clearly $\Phi(\cJ)=\cK$ and for
$f\in C(\dot{\R})$ and $g\in SO_p$,
\[
\Phi((fI)+\cJ)=fI+\cK, \quad \Phi((W^0(g))+\cJ)=W^0(g)+\cK.
\]
Hence the mapping $\Phi|_{\cC^\cJ}:\cC^\cJ\to\cC^\pi$ is
injective and onto, where the algebra $\cC^\pi$ is defined just before
Lemma~\ref{le:maximal-ideal-Cpi}. Therefore $\cC^\cJ$ and
$\cC^\pi$ are isomorphic. Then the maximal ideal spaces of
$\cC^\cJ$ and $\cC^\pi$ are homeomorphic. Thus, from
Lemma~\ref{le:maximal-ideal-Cpi} it follows that $M(\cC^\cJ)$ is
homeomorphic to the set $\Omega$ given by \eqref{eq:def-Omega}.
\end{proof}
\subsection{Localization via the Allan-Douglas local principle}
\label{subsec:localization-our}
With every point $(\xi,\eta)\in\Omega$ we associate the closed two-sided
ideal $\cN_{\xi,\eta}^\cJ$ of the algebra $\cC^\cJ$ generated by all the
cosets
\[
(f I)+\cJ \quad (f\in C(\dot{\R}),\ f(\xi)=0),
\quad\quad
(W^0(g))+\cJ \quad (g\in SO_p,\ g(\eta)=0).
\]
Let $\cI_{\xi,\eta}^\cJ$ be the smallest closed two-sided ideal of $\cL^\cJ$
that contains the maximal ideal $\cN_{\xi,\eta}^\cJ$ and let
\[
\Phi_{\xi,\eta}^\cJ:\cL^\cJ\to\cL^\cJ/\cI_{\xi,\eta}^\cJ=:\cL_{\xi,\eta}^\cJ
\]
be the canonical homomorphism of $\cL^\cJ$ onto the quotient
algebra $\cL_{\xi,\eta}^\cJ$. Now we are in a position to apply
the Allan-Douglas local principle. Summarizing the results
obtained so far we arrive at the following.
\begin{theorem}\label{th:localization}
A sequence $\bA=(A_\tau)\in\cL$ is stable if and only if the
operators $\fW_{-1}(\bA)$, $\fW_0(\bA)$, and $\fW_1(\bA)$ are
invertible in $\cB$ and the coset $\bA+\cI_{\xi,\eta}^\cJ$ is
invertible in the local algebra $\cL_{\xi,\eta}^\cJ$ for every
$(\xi,\eta)\in\Omega$, where $\Omega$ is given by
\eqref{eq:def-Omega}.
\end{theorem}
\begin{proof}
From Lemma~\ref{le:alg-F}(a) and the definition of $\cL$ we know
that $\cL\subset\cF\subset\cE$. In view of Theorem~\ref{th:Kozak},
$\bA$ is stable if and only if $\bA+\cG$ is invertible in
$\cE/\cG$. By Lemma~\ref{le:alg-F}(e), the latter is
equivalent to the invertibility of $\bA+\cG$ in $\cF/\cG$. By the
lifting theorem (Theorem~\ref{th:lifting-our}), this is equivalent
to the invertibility of the operators $\fW_{-1}(\bA)$,
$\fW_0(\bA)$, and $\fW_1(\bA)$ in the algebra $\cB$ and the
invertibility of the coset $\bA+\cJ$ in the quotient algebra
$\cF/\cJ$. From Lemma~\ref{le:alg-L}(c) we conclude that $\bA+\cJ$
is invertible in $\cF/\cJ$ if and only if it is invertible in
$\cL/\cJ$. The latter assertion is equivalent to the invertibility
of the cosets $\bA+\cI_{\xi,\eta}^\cJ$ in the local algebras
$\cL_{\xi,\eta}^\cJ$ for all $(\xi,\eta)\in\Omega$ by the
Allan-Douglas local principle (Theorem~\ref{th:AllanDouglas}).
\end{proof}
Our next aim is to study the invertibility of the elements
\[
\Phi_{\xi,\eta}^\cJ(\bA):=\bA+\cI_{\xi,\eta}^\cJ \quad (\bA\in\cL)
\]
in the local algebras $\cL_{\xi,\eta}^\cJ$. It turns out that the
algebras $\cL_{\xi,\eta}^\cJ$ are too large in order to obtain via
Theorem~\ref{th:localization} effectively verifiable stability
conditions for an arbitrary $\bA\in\cL$. So we restrict ourselves
to the case of sequences belonging to $\cA$ (recall that
$\cA\subset\cL$ by Theorem~\ref{th:A-local-type}). We denote by
$\cA^\cJ$ the smallest closed subalgebra of $\cL^\cJ$ that
contains the cosets $\bA+\cJ$ for all $\bA=(A_\tau)\in\cA$. The
canonical homomorphism $\Phi_{\xi,\eta}^\cJ$ sends $\cA^\cJ$ onto
\[
\cA_{\xi,\eta}^\cJ:=\Phi_{\xi,\eta}^\cJ(\cA^\cJ)\subset\cL_{\xi,\eta}^\cJ.
\]
In
Sections~\ref{section:identification-1}--\ref{section:identification-3}
we obtain sufficient conditions for the invertibility of the
elements $\Phi_{\xi,\eta}^\cJ(\bA)$ in the local algebras
$\cL_{\xi,\eta}^\cJ$ for $\bA\in\cA$ and $(\xi,\eta)\in\Omega$.
\section{Homogenization}\label{section:homogenization}
\subsection{Algebra of sequences generated by homogeneous operators and
\boldmath{$(P_\tau)$}}
\label{subsection:homo}
Given any positive real number $\tau$, define the operator
\[
Z_\tau:L^p(\R)\to L^p(\R), \quad (Z_\tau f)(x)=\tau^{-1/p}f(x/\tau).
\]
Obviously
$\|Z_\tau\|_\cB=1$. It is also clear that $Z_\tau$ is
invertible and $Z_\tau^{-1}=Z_{1/\tau}$. An operator $A\in\cB$ is
called \textit{homogeneous} (or \textit{dilation invariant}) if
$Z_\tau^{-1}AZ_\tau=A$ for each $\tau\in\R_+$. It is easy to see
that the operators $\chi_\pm I$ and $W^0(\chi_\pm)$ are homogeneous
operators.

Let $\cH$ be the smallest closed subalgebra of
$\cE$ that contains the sequence $(P_\tau)$ and the constant
sequences $(H)$, where $H$ is a homogeneous operator.
\begin{lemma}\label{le:homogeneous}
Let $\bA=(A_\tau)\in\cH\cap\cL$. Then the coset $\bA+\cG$ is
invertible in the algebra $\cL/\cG$ if and only if the operator
$A_1$ is invertible.
\end{lemma}
\begin{proof}
The idea of the proof is borrowed from
\cite[Proposition~1]{Roch88}.

\textit{Necessity.} If $\bA+\cG$ is invertible in $\cL/\cG$, then
obviously it is invertible in $\cE/\cG$. Therefore, by
Theorem~\ref{th:Kozak}, $\bA$ is stable. Hence the operators
$A_\tau$ are invertible for all sufficiently large $\tau$. Since
\begin{equation}\label{eq:homogeneous}
Z_\tau^{-1}A_\tau Z_\tau=A_1
\quad
(\tau\in\R_+),
\end{equation}
we conclude that the operator $A_1$ is invertible. The necessity
portion is proved.

\textit{Sufficiency.} If $A_1$ is invertible, then
\eqref{eq:homogeneous} implies that $A_\tau^{-1}=Z_\tau
A_1^{-1}Z_\tau^{-1}$. Since $\|Z_\tau\|_\cB\|Z_\tau^{-1}\|_\cB=1$,
the norms of $A_\tau^{-1}$ are uniformly bounded. That is
$\bA=(A_\tau)$ is stable. By Theorem~\ref{th:Kozak}, $\bA+\cG$ is
invertible in $\cE/\cG$. But $\cL/\cG$ is inverse closed in
$\cF/\cG$ (Lemma~\ref{le:alg-L}(c)) and $\cF/\cG$ is inverse
closed in $\cE/\cG$ (Lemma~\ref{le:alg-F}(e)). Thus $\bA+\cG$
is invertible in $\cL/\cG$.
\end{proof}
\subsection{Algebras \boldmath{$\cH_\eta$} and their ideal \boldmath{$\cG$}}
For $\eta\in\R$, put
\[
U_\eta:L^p(\R)\to L^p(\R),\quad (U_\eta f)(x)=e^{i\eta x}f(x).
\]
It is clear that $U_\eta^{-1}=U_{-\eta}$ and $\|U_\eta^{\pm 1}\|_\cB=1$.

Let $\cH_\eta$ denote the set of all sequences $\bA=(A_\tau)\in\cE$ such
that the sequence $(Z_\tau^{-1}U_\eta A_\tau U_\eta^{-1}Z_\tau)$ is
*-strongly convergent as $\tau\to+\infty$.
\begin{lemma}\label{le:alg-Heta}
Suppose $\eta\in\R$.
\begin{enumerate}
\item[{\rm(a)}]
The set $\cH_\eta$ is a closed unital subalgebra of the algebra $\cE$.

\item[{\rm (b)}]
The mapping $\fH_\eta:\cH_\eta\to\cB$ given by
\[
\fH_\eta(\bA):=\operatornamewithlimits{s-lim}_{\tau\to+\infty}
Z_\tau^{-1}U_\eta A_\tau U_\eta^{-1}Z_\tau
\]
for $\bA=(A_\tau)\in\cH_\eta$ is a bounded unital homomorphism with the
norm
\[
\|\fH_\eta\|=1.
\]

\item[{\rm(c)}]
The set $\cG$ is a closed two-sided ideal of the algebra $\cH_\eta$.

\item[{\rm (d)}]
The ideal $\cG$ lies in the kernel of the homomorphism
$\fH_\eta$.

\item[{\rm (e)}]
The algebra $\cH_\eta/\cG$ is inverse closed in
the algebra $\cE/\cG$.
\end{enumerate}
\end{lemma}
The proof is analogous to the proof of \cite[Proposition~4.1]{RSS09}.
\begin{remark}\label{rem:algebras-F}
Let $\widetilde{\cF}$ be the algebra denoted by $\cF$ in the paper
\cite{RSS09}. Then
\[
\widetilde{\cF}\subset\cF\cap\left(\cap_{\eta\in\R}\cH_\eta\right).
\]
We conjecture that this inclusion is proper.
\end{remark}
\subsection{The algebra \boldmath{$\cA$} is contained in the algebras
\boldmath{$\cH_\eta$}}
\begin{proposition}\label{pr:A-in-Heta}
Suppose $\eta\in\R$.
\begin{enumerate}
\item[{\rm(a)}]
If $\bP=(P_\tau)$, then $\bP\in\cH_\eta$ and $\fH_\eta(\bP)=P_1$.

\item[{\rm(b)}]
If $\bA=(aI)$ with $a\in PC$, then $\bA\in\cH_\eta$ and
\[
\fH_\eta(\bA)=a(-\infty)\chi_-I+a(+\infty)\chi_+I.
\]

\item[{\rm(c)}]
If $\bB=(W^0(b))$ with $b\in[PC_p,SO_p]$, then $\bB\in\cH_\eta$ and
\[
\fH_\eta(\bB)=b(\eta^-)W^0(\chi_-)+b(\eta^+)W^0(\chi_+).
\]
\end{enumerate}
\end{proposition}
\begin{proof}
(a) It is easy to see that for every $\tau\in\R_+$ one has
$U_\eta P_\tau U_\eta^{-1}=P_\tau$ and $Z_\tau^{-1}P_\tau Z_\tau=P_1$. From these
equalities it follows immediately that $\fH_\eta(\bP)=P_1$.
Taking into account that
\[
Z_\tau^*=Z_{1/\tau}, \quad U_\eta^*=U_{-\eta}, \quad P_\tau^*=P_\tau,
\]
from above we conclude that $(Z_\tau^{-1}U_\eta P_\tau U_\eta^{-1}Z_\tau)^*$
converges strongly to $P_1^*$. Thus $\bP\in\cH_\eta$.
Part (a) is proved.

Part (b) is proved in \cite[Proposition~13.1(b)]{RS90} and the
proof of part (c) can be developed by analogy with
\cite[Theorem~4.2(i)]{BBK04-MN}.
\end{proof}
\begin{remark}
For $\bA\in\cA$, each operator $\fH_\eta(\bA)$ is homogeneous. So
the passage from a sequence $\bA\in\cA$ to its image under the
homomorphism $\fH_\eta$ can be naturally called
\textit{homogenization}.
\end{remark}
\begin{corollary}
For every $\eta\in\R$, the algebra $\cA$ is a closed unital subalgebra of
the algebra $\cH_\eta$.
\end{corollary}
\subsection{Necessary condition for the stability of \boldmath{$\bA\in\cA$}}
Now we are in a position to prove the main result of this section.
\begin{theorem}\label{th:nec-stability-Heta}
If a sequence $\bA\in\cA$ is stable, then the operators $\fH_\eta(\bA)$ are
invertible in the algebra $\cB$ for all $\eta\in\R$.
\end{theorem}
\begin{proof}
If $\bA\in\cA$ is stable, then by Theorem~\ref{th:Kozak}, the coset
$\bA+\cG$ is invertible in the quotient algebra $\cE/\cG$. Then from
Lemma~\ref{le:alg-Heta}(e) it follows that $\bA+\cG$ is invertible
in the quotient algebra $\cH_\eta/\cG$ for every $\eta\in\R$. This means
that there exist $\bB\in\cH_\eta$ and $\bG_1,\bG_2\in\cG$ such that
\begin{equation}\label{eq:nec-stability-Heta}
\bA\bB=\bI+\bG_1,\quad \bB\bA=\bI+\bG_2.
\end{equation}
From Lemma~\ref{le:alg-Heta}(b),(d) we know that $\fH_\eta:\cH_\eta\to\cB$
is a unital homomorphism and $\fH_\eta(\bG_1)=\fH_\eta(\bG_2)=0$.
Applying this homomorphism to the equalities in \eqref{eq:nec-stability-Heta},
we obtain that $\fH_\eta(\bA)$ is invertible in $\cB$ and its inverse is equal
to $\fH_\eta(\bB)$.
\end{proof}
\begin{remark}
The proof of Theorem~\ref{th:nec-stability-Heta} given above does not use
the results of Sections~\ref{section:essentialization} and \ref{section:localization}.
\end{remark}
\subsection{Strong convergence of families of sequences associated to the
fiber \boldmath{$M_\infty(SO)$}} The following statement can be
considered as a counterpart of Proposition~\ref{pr:A-in-Heta}(c)
for the fiber $M_\infty(SO)$. It will be used in
Section~\ref{section:identification-3} and can be proved by
analogy with \cite[Theorem~4.2(ii)]{BBK04-MN} with the aid of
Lemma~\ref{le:SO-partial-limits}.
\begin{lemma}\label{le:family-convergence}
Suppose $\eta\in M_\infty(SO)$ and $g_1,\dots,g_m\in SO_p$ is a
finite family such that $\eta(g_1)=\dots =\eta(g_m)=0$. Then there
exists a sequence $(\tau_n)_{n=1}^\infty\subset\R_+$ such that
$\tau_n\to +\infty$ as $n\to\infty$ and
\[
\operatornamewithlimits{s-lim}_{n\to\infty}
Z_{\tau_n}W^0(g_j)Z_{\tau_n}^{-1}=0
\quad\mbox{for all}\quad j=1,\dots,m.
\]
\end{lemma}
\section{Invertibility in the local algebra \boldmath{$\cL_{\xi,\eta}^\cJ$}
with \boldmath{$(\xi,\eta)\in\R\times M_\infty(SO)$}}
\label{section:identification-1}
\subsection{Local images of elements of \boldmath{$\cA$}}
\begin{lemma}\label{le:case-1}
Suppose $\bA\in\cA$ and $(\xi,\eta)\in\R\times M_\infty(SO)$. Then
the constant sequence $\bA_0:=(\fW_0(\bA))$ belongs to $\cA$ and
\begin{equation}\label{eq:case1-1}
\Phi_{\xi,\eta}^\cJ(\bA)=\Phi_{\xi,\eta}^\cJ(\bA_0).
\end{equation}
\end{lemma}
\begin{proof}
This statement is proved by analogy with \cite[Proposition~2.14]{RSS97}.
If $\bA=(aI)$ with $a\in PC$ or $\bA=(W^0(b))$ with
$b\in[PC_p,SO_p]$, then obviously $\bA_0=\bA\in\cA$. For
$\bP=(P_\tau)\in\cA$, we have $\fW_0(\bP)=I$ and
$\bP_0=(I)=:\bI\in\cA$. These facts imply that $\bA_0\in\cA$.

For all constant sequences $\bA\in\cA$, we have $\bA=\bA_0\in\cA$. This
implies \eqref{eq:case1-1} for such sequences. It remains to show that
$\Phi_{\xi,\eta}^\cJ(\bP)=\Phi_{\xi,\eta}^\cJ(\bI)$.

Assume that $y>|\xi|$ and consider a function $f_\xi\in
C(\dot{\R})$ such that $f_\xi(\xi)=1$ and
$\operatorname{supp}f_\xi\subset(-y,y)$. From the definition of
the ideal $\cI_{\xi,\eta}^\cJ$ it follows that
\[
(f_\xi I)-(I)+\cJ=((f_\xi -1)I)+\cJ\in\cI_{\xi,\eta}^\cJ.
\]
Hence $\Phi_{\xi,\eta}^\cJ[(f_\xi I)]=\Phi_{\xi,\eta}^\cJ(\bI)$ is
the identity in the local algebra $\cL_{\xi,\eta}^\cJ$. Put
$\bQ:=(Q_\tau)$. Then
\[
\Phi_{\xi,\eta}^\cJ(\bQ) = \Phi_{\xi,\eta}^\cJ[(f_\xi
I)]\Phi_{\xi,\eta}^\cJ(\bQ) = \Phi_{\xi,\eta}^\cJ[(f_\xi Q_\tau)].
\]
Since $\operatorname{supp}f_\xi\subset(-y,y)$, we have $\|f_\xi
Q_\tau\|_\cB\to 0$ as $\tau\to\infty$. Then $(f_\xi
Q_\tau)\in\cG\subset\cJ$. This means that
$\Phi_{\xi,\eta}^\cJ(\bQ)$ is the zero in $\cL_{\xi,\eta}^\cJ$.
Hence $\Phi_{\xi,\eta}^\cJ(\bP)=\Phi_{\xi,\eta}^\cJ(\bI)$.
\end{proof}
\subsection{Sufficient conditions for the invertibility in \boldmath{$\cL_{\xi,\eta}^\cJ$}}
\begin{theorem}\label{th:case-1}
Let $(\xi,\eta)\in\R\times M_\infty(SO)$. If $\bA\in\cA$ and the
operator $\fW_0(\bA)$ is invertible, then the coset
$\Phi_{\xi,\eta}^\cJ(\bA)$ is invertible in the local algebra
$\cL_{\xi,\eta}^\cJ$.
\end{theorem}
\begin{proof}
If $\fW_0(\bA)$ is invertible, then the constant sequence
$\bA_0=(\fW_0(\bA))\in\cA$ is stable. From Theorem~\ref{th:Kozak}
and Lemma~\ref{le:alg-F}(e) we obtain that $\bA_0+\cG$ is
invertible in $\cF/\cG$. Therefore, $\bA_0+\cJ$ is invertible in
the quotient algebra $\cF/\cJ$ because $\cG\subset\cJ$. From
Lemma~\ref{le:alg-L}(c) it follows that the latter fact is
equivalent to the invertibility of $\bA_0+\cJ$ in the quotient
algebra $\cL/\cJ$. By Theorem~\ref{th:AllanDouglas}, this implies
that $\Phi_{\xi,\eta}^\cJ(\bA_0)$ is invertible in the local
algebra $\cL_{\xi,\eta}^\cJ$. It remains to recall that
$\Phi_{\xi,\eta}^\cJ(\bA)=\Phi_{\xi,\eta}^\cJ(\bA_0)$ by
Lemma~\ref{le:case-1}(b).
\end{proof}
\section{Invertibility in the local algebra \boldmath{$\cL_{\infty,\eta}^\cJ$}
with \boldmath{$\eta\in\R$}}
\label{section:identification-2}
\subsection{Two auxiliary lemmas}
Given $\eta\in\R$, let $\chi_\eta^-$ and $\chi_\eta^+$ denote the
characteristic functions of $(-\infty,\eta)$ and $(\eta,+\infty)$,
respectively. Clearly, $\chi_0^-=\chi_-$ and $\chi_0^+=\chi_+$. It is
easy to see that
\[
W^0(\chi_\eta^-)=U_\eta^{-1} W^0(\chi_-)U_\eta,\quad
W^0(\chi_\eta^+)=U_\eta^{-1} W^0(\chi_+)U_\eta.
\]
Consider the constant sequences
\[
\bX_-:=(\chi_-I),\quad
\bX_+:=(\chi_+I),\quad
\bW_-^\eta:=(W^0(\chi_\eta^-)),\quad
\bW_+^\eta:=(W^0(\chi_\eta^+)).
\]
\begin{lemma}\label{le:case-2-repr2}
Let $\eta\in\R$.
\begin{enumerate}
\item[{\rm(a)}]
If $a\in PC$ and $\bA=(aI)$, then
\[
\Phi_{\infty,\eta}^\cJ(\bA)=
a(-\infty)\Phi_{\infty,\eta}^\cJ(\bX_-)+
a(+\infty)\Phi_{\infty,\eta}^\cJ(\bX_+).
\]

\item[{\rm (b)}]
If $b\in [PC_p,SO_p]$ and $\bB=(W^0(b))$, then
\[
\Phi_{\infty,\eta}^\cJ(\bB)=
b(\eta^-)\Phi_{\infty,\eta}^\cJ(\bW_-^\eta)+
b(\eta^+)\Phi_{\infty,\eta}^\cJ(\bW_+^\eta).
\]
\end{enumerate}
\end{lemma}
\begin{proof}
(a)
Consider $a_0:=a-[a(-\infty)\chi_-+a(+\infty)\chi_+]$. Clearly $a_0\in PC$
and $a_0(-\infty)=a_0(+\infty)=0$. Therefore $(a_0I)+\cJ\in\cI_{\infty,\eta}^\cJ$.
Hence
\[
0=\Phi_{\infty,\eta}^\cJ[(a_0I)]=
\Phi_{\infty,\eta}^\cJ(\bA)-[a(-\infty)\Phi_{\infty,\eta}^\cJ(\bX_-)+a(+\infty)\Phi_{\infty,\eta}^\cJ(\bX_+)].
\]
Part (a) is proved.

(b) The function
$b_0:=b-[b(\eta^-)\chi_\eta^-+b(\eta^+)\chi_\eta^+]$ belongs to
$[PC_p,SO_p]$ and, moreover, it is continuous and vanishing at the
point $\eta\in\R$. Hence the coset $(W^0(b_0))+\cJ$ belongs to the
ideal $\cI_{\infty,\eta}^\cJ$. Therefore
\[
0=\Phi_{\infty,\eta}^\cJ[(W^0(b_0I))]=
\Phi_{\infty,\eta}^\cJ(\bB)-
[b(\eta^-)\Phi_{\infty,\eta}^\cJ(\bW_-^\eta)+b(\eta^+)\Phi_{\infty,\eta}^\cJ(\bW_+^\eta)].
\]
This completes the proof of part (b).
\end{proof}
\begin{lemma}\label{le:identity}
Suppose $\bA\in\cA$ and $\eta\in\R$. Then the sequence $\bA_\eta$ given by
\[
\bA_\eta:=(U_\eta^{-1} Z_\tau\fH_\eta(\bA)Z_\tau^{-1}U_\eta)
\]
belongs to $\cA$ and
\begin{equation}\label{eq:identity-1}
\Phi_{\infty,\eta}^\cJ(\bA)=
\Phi_{\infty,\eta}^\cJ(\bA_\eta).
\end{equation}
\end{lemma}
\begin{proof}
From Proposition~\ref{pr:A-in-Heta}(a) it follows that
\begin{equation}\label{eq:identity-2}
U_\eta^{-1} Z_\tau\fH_\eta(\bP)Z_\tau^{-1}U_\eta =
U_\eta^{-1} Z_\tau P_1 Z_\tau^{-1}U_\eta=P_\tau.
\end{equation}
By Proposition~\ref{pr:A-in-Heta}(b), for $\bA=(aI)$ with $a\in PC$ we have
\begin{eqnarray}
U_\eta^{-1} Z_\tau\fH_\eta(\bA)Z_\tau^{-1}U_\eta
&=&
U_\eta^{-1} Z_\tau[a(-\infty)\chi_-I+a(+\infty)\chi_+I]Z_\tau^{-1} U_\eta
\nonumber
\\
&=&
a(-\infty)Z_\tau\chi_-Z_\tau^{-1}+a(+\infty)Z_\tau\chi_+Z_\tau^{-1}
\nonumber
\\
&=&
a(-\infty)\chi_-I+a(+\infty)\chi_+I.
\label{eq:identity-3}
\end{eqnarray}
Assume that $b\in[PC_p,SO_p]$ and $\bA=(W^0(b))$. By Proposition~\ref{pr:A-in-Heta}(c),
\begin{eqnarray}
&& U_\eta^{-1} Z_\tau\fH_\eta(\bA)Z_\tau^{-1}U_\eta
=
U_\eta^{-1} Z_\tau[b(\eta^-)W^0(\chi_-)+b(\eta^+)W^0(\chi_+)]Z_\tau^{-1} U_\eta
\nonumber
\\
&&=
b(\eta^-)U_\eta^{-1} Z_\tau W^0(\chi_-)Z_\tau^{-1}U_\eta+
b(\eta^+)U_\eta^{-1} Z_\tau W^0(\chi_+)Z_\tau^{-1}U_\eta
\nonumber
\\
&&=
b(\eta^-)U_\eta^{-1} W^0(\chi_-)U_\eta+
b(\eta^+)U_\eta^{-1} W^0(\chi_+)U_\eta
\nonumber
\\
&&=b(\eta^-)W^0(\chi_\eta^-)+b(\eta^+)W^0(\chi_\eta^+).
\label{eq:identity-4}
\end{eqnarray}
Equalities \eqref{eq:identity-2}--\eqref{eq:identity-4} imply that
$\bA_\eta\in\cA$ for every $\bA\in\cA$.

Obviously, it is sufficient to check identity \eqref{eq:identity-1}
on the generators of the algebra $\cA$. From \eqref{eq:identity-2}
we get
\[
\Phi_{\infty,\eta}^\cJ(\bA_\eta)=\Phi_{\infty,\eta}^\cJ(\bA)
\]
for $\bA=\bP$. From Lemma~\ref{le:case-2-repr2}(a) and equality
\eqref{eq:identity-3} it follows that
\[
\Phi_{\infty,\eta}^\cJ(\bA_\eta)
=
a(-\infty)\Phi_{\infty,\eta}^\cJ(\bX_-)+a(+\infty)\Phi_{\infty,\eta}^\cJ(\bX_+)
=
\Phi_{\infty,\eta}^\cJ(\bA)
\]
for $\bA=(aI)$ with $a\in PC$.
Equality \eqref{eq:identity-4} and Lemma~\ref{le:case-2-repr2}(b) yield
\[
\Phi_{\infty,\eta}^\cJ(\bA_\eta) =
b(\eta^-)\Phi_{\infty,\eta}^\cJ(\bW_-^\eta)+b(\eta^+)\Phi_{\infty,\eta}^\cJ(\bW_+^\eta)
= \Phi_{\infty,\eta}^\cJ(\bA)
\]
with $\bA=(W^0(b))$ and $b\in[PC_p,SO_p]$.
\end{proof}
\subsection{Sufficient condition for the invertibility in \boldmath{$\cL_{\infty,\eta}^\cJ$}}
\begin{theorem}\label{th:case-2}
Let $\eta\in\R$ and $\bA\in\cA$. If the operator
$\fH_\eta(\bA)$ is invertible in the algebra $\cB$, then the coset
$\Phi_{\infty,\eta}^\cJ(\bA)$ is invertible in the local algebra $\cL_{\infty,\eta}^\cJ$.
\end{theorem}
\begin{proof}
Let $B\in\cB$ be the inverse of $\fH_\eta(\bA)$. Multiplying the equalities
\[
B\fH_\eta(\bA)=I=\fH_\eta(\bA)B
\]
from the left by $U_\eta^{-1} Z_\tau$ and from the right by $Z_\tau^{-1}U_\eta$,
we obtain
\[
\bB_{(\eta)}\bA_\eta=\bI=\bA_\eta\bB_{(\eta)},
\]
where $\bB_{(\eta)}:=(U_\eta^{-1} Z_\tau BZ_\tau^{-1}U_\eta)\in\cE$.
It is easy to see that $\|\bB_{(\eta)}\|_\cE\le\|B\|_\cB$. Hence all
operators $U_\eta^{-1} Z_\tau\fH_\eta(\bA)Z_\tau^{-1}U_\eta$ are
invertible and the norms of their inverses are uniformly bounded.
Thus the sequence $\bA_\eta$ is stable. From Lemma~\ref{le:identity}
and Theorem~\ref{th:A-local-type} we conclude that $\bA_\eta\in\cL$.
Then the stability of $\bA_\eta$ implies the invertibility of
the coset $\Phi_{\infty,\eta}^\cJ(\bA_\eta)$ in the local algebra $\cL_{\infty,\eta}^\cJ$
due to Theorem~\ref{th:localization}. It remains to recall that
$\Phi_{\infty,\eta}^\cJ(\bA_\eta)=\Phi_{\infty,\eta}^\cJ(\bA)$
by Lemma~\ref{le:identity}.
\end{proof}
\section{Invertibility in the local algebra \boldmath{$\cL_{\infty,\eta}^\cJ$}
with \boldmath{$\eta\in M_\infty(SO)$}}
\label{section:identification-3}
\subsection{On the center of the algebra \boldmath{$\cA_{\infty,\eta}^\cJ$}}
Let $\bP:=(P_\tau)$ and $\bQ:=(Q_\tau)$. From
Lemma~\ref{le:Simonenko} it follows that the functions
\[
u_\pm(x):=(1\pm\tanh x)/2\quad(x\in\overline{\R})
\]
belong to $C_p(\overline{\R})$. Hence the convolution operators
$W^0(u_\pm)$ are well defined and bounded on $L^p(\R)$. Consider
the following constant sequences
\[
\begin{array}{llll}
\bU_-:=(u_-I), &
\bU_+:=(u_+I), &
\bV_-:=(W^0(u_-)), &
\bV_+:=(W^0(u_+)),
\\[2mm]
\bX_-:=(\chi_-I), &
\bX_+:=(\chi_+I), &
\bW_-:=(W^0(\chi_-)), &
\bW_+:=(W^0(\chi_+)).
\end{array}
\]
\begin{lemma}\label{le:case-3-repr}
Let $\eta\in M_\infty(SO)$.
\begin{enumerate}
\item[{\rm (a)}]
We have
\[
\begin{array}{ll}
\Phi_{\infty,\eta}^\cJ(\bX_-)=\Phi_{\infty,\eta}^\cJ(\bU_-), &
\Phi_{\infty,\eta}^\cJ(\bX_+)=\Phi_{\infty,\eta}^\cJ(\bU_+),\\[2mm]
\Phi_{\infty,\eta}^\cJ(\bW_-)=\Phi_{\infty,\eta}^\cJ(\bV_-), &
\Phi_{\infty,\eta}^\cJ(\bW_+)=\Phi_{\infty,\eta}^\cJ(\bV_+).
\end{array}
\]

\item[{\rm (b)}]
If $a\in PC$ and $\bA=(aI)$, then
\[
\Phi_{\infty,\eta}^\cJ(\bA)=
a(-\infty)\Phi_{\infty,\eta}^\cJ(\bX_-)+
a(+\infty)\Phi_{\infty,\eta}^\cJ(\bX_+).
\]

\item[{\rm (c)}]
If $b\in [PC_p,SO_p]$ and $\bB=(W^0(b))$, then
\[
\Phi_{\infty,\eta}^\cJ(\bB)=
b_\eta(-\infty)\Phi_{\infty,\eta}^\cJ(\bW_-)+
b_\eta(+\infty)\Phi_{\infty,\eta}^\cJ(\bW_+).
\]
\end{enumerate}
\end{lemma}
\begin{proof}
(a) Since $\chi_-(-\infty)=u_-(-\infty)$ and $\chi_-(+\infty)=u_-(+\infty)$,
the function $\chi_--u_-\in PC_p$ is continuous and vanishing at the point
$\infty$ of $\dot{\R}$. Hence the cosets $\bX_--\bU_-+\cJ$ and $\bW_--\bV_-+\cJ$
belong to the ideal $\cI_{\infty,\eta}^\cJ$. Therefore
$\Phi_{\infty,\eta}^\cJ(\bX_-)=\Phi_{\infty,\eta}^\cJ(\bU_-)$ and
$\Phi_{\infty,\eta}^\cJ(\bW_-)=\Phi_{\infty,\eta}^\cJ(\bV_-)$. Analogously one
can show that
$\Phi_{\infty,\eta}^\cJ(\bX_+)=\Phi_{\infty,\eta}^\cJ(\bU_+)$ and
$\Phi_{\infty,\eta}^\cJ(\bW_+)=\Phi_{\infty,\eta}^\cJ(\bV_+)$.

(b) The proof coincides with the proof of Lemma~\ref{le:case-2-repr2}(a).

(c) Consider the function
$\widetilde{b}:=b-b_\eta(-\infty)\chi_--b_\eta(+\infty)\chi_+\in
[PC_p,SO_p]$. Then
$(\alpha_\eta\widetilde{b})(-\infty)=(\alpha_\eta\widetilde{b})(+\infty)=0$.
In particular, if $b\in SO_p$, then $\widetilde{b}=b-b(\eta)$ and $\widetilde{b}(\eta)=0$.
This implies that
$(W^0(\widetilde{b}))+\cJ\in\cI_{\infty,\eta}^\cJ$ for $b\in[PC_p,SO_p]$. Thus
\[
0=\Phi_{\infty,\eta}^\cJ[(W^0(\widetilde{b}))]=\Phi_{\infty,\eta}^\cJ(\bB)-
b_\eta(-\infty)\Phi_{\infty,\eta}^\cJ(\bX_-)-b_\eta(+\infty)\Phi_{\infty,\eta}^\cJ(\bX_+),
\]
which finishes the proof of part (c).
\end{proof}
\begin{lemma}\label{le:commuting-idempotent}
Let $\eta\in M_\infty(SO)$. Then the cosets
$\Phi_{\infty,\eta}^\cJ(\bX_-)$ and $\Phi_{\infty,\eta}^\cJ(\bX_+)$
commute with any element of the local algebra
$\cA_{\infty,\eta}^\cJ$.
\end{lemma}
\begin{proof}
We will prove that $\Phi_{\infty,\eta}^\cJ(\bX_+)$ commutes with the
elements of $\cA_{\infty,\eta}^\cJ$. It is sufficient to prove this
statement for the generators of $\cA_{\infty,\eta}^\cJ$. It is
obvious that $\chi_+P_\tau=P_\tau\chi_+I$ and $\chi_+a=a\chi_+$
for all $a\in PC$. As usual, denote $\bP=(P_\tau)$ and $\bA=(aI)$.
Then
\[
\begin{split}
&
\Phi_{\infty,\eta}^\cJ(\bX_+)\Phi_{\infty,\eta}^\cJ(\bP)=
\Phi_{\infty,\eta}^\cJ(\bP)\Phi_{\infty,\eta}^\cJ(\bX_+),
\\
&
\Phi_{\infty,\eta}^\cJ(\bX_+)\Phi_{\infty,\eta}^\cJ(\bA)=
\Phi_{\infty,\eta}^\cJ(\bA)\Phi_{\infty,\eta}^\cJ(\bX_+).
\end{split}
\]
By Lemma~\ref{le:case-3-repr}(a),(c),
\begin{eqnarray}
&&
\Phi_{\infty,\eta}^\cJ(\bX_+)
=
\Phi_{\infty,\eta}^\cJ(\bU_+),
\label{eq:commuting-idempotent-1}
\\
&&
\Phi_{\infty,\eta}^\cJ(\bB)
=
b_\eta(-\infty)\Phi_{\infty,\eta}^\cJ(\bV_-)+
b_\eta(+\infty)\Phi_{\infty,\eta}^\cJ(\bV_+).
\label{eq:commuting-idempotent-2}
\end{eqnarray}
In view of Lemma~\ref{le:Duduchava},
$u_+W^0(u_-)-W^0(u_-)u_+I\in\cK$. Hence
$\bU_+\bV_--\bV_-\bU_+\in\cJ$ and therefore
\begin{equation}\label{eq:commuting-idempotent-3}
\Phi_{\infty,\eta}^\cJ(\bU_+\bV_-)=\Phi_{\infty,\eta}^\cJ(\bV_-\bU_+).
\end{equation}
Analogously,
\begin{equation}\label{eq:commuting-idempotent-4}
\Phi_{\infty,\eta}^\cJ(\bU_+\bV_+)=\Phi_{\infty,\eta}^\cJ(\bV_+\bU_+).
\end{equation}
Combining
\eqref{eq:commuting-idempotent-1}--\eqref{eq:commuting-idempotent-4},
we get
\[
\begin{split}
\Phi_{\infty,\eta}^\cJ(\bX_+)\Phi_{\infty,\eta}^\cJ(\bB)
&=
\Phi_{\infty,\eta}^\cJ(\bU_+)
\big[b_\eta(-\infty)\Phi_{\infty,\eta}^\cJ(\bV_-)+
b_\eta(+\infty)\Phi_{\infty,\eta}^\cJ(\bV_+)\big]
\\
&=
b_\eta(-\infty)\Phi_{\infty,\eta}^\cJ(\bU_+\bV_-)+
b_\eta(+\infty)\Phi_{\infty,\eta}^\cJ(\bU_+\bV_+)
\\
&=
b_\eta(-\infty)\Phi_{\infty,\eta}^\cJ(\bV_-\bU_+)+
b_\eta(+\infty)\Phi_{\infty,\eta}^\cJ(\bV_+\bU_+)
\\
&=
\big[b_\eta(-\infty)\Phi_{\infty,\eta}^\cJ(\bV_-)+
b_\eta(+\infty)\Phi_{\infty,\eta}^\cJ(\bV_+)\big]
\Phi_{\infty,\eta}^\cJ(\bU_+)
\\
&=
\Phi_{\infty,\eta}^\cJ(\bB)\Phi_{\infty,\eta}^\cJ(\bX_+),
\end{split}
\]
which finishes the proof of the statement for
$\Phi_{\infty,\eta}^\cJ(\bX_+)$. This immediately implies that
$\Phi_{\infty,\eta}^\cJ(\bX_-)$ commutes with the elements of
$\cA_{\infty,\eta}^\cJ$ because
$\Phi_{\infty,\eta}^\cJ(\bX_-)=\Phi_{\infty,\eta}^\cJ(\bI)-\Phi_{\infty,\eta}^\cJ(\bX_+)$
and $\Phi_{\infty,\eta}^\cJ(\bI)$ is the identity of
$\cA_{\infty,\eta}^\cJ$.
\end{proof}
\subsection{Reduction to algebras generated by two idempotents}
An element $p$ of a Banach algebra $B$ is said to be
\textit{idempotent} if $p^2=p$. If $B$ is a unital Banach algebra
with identity $e$, then $e-p$ is also an idempotent and 
\[
pBp:=\big\{pap:a\in B\big\}, \quad
(e-p)B(e-p):=\big\{(e-p)a(e-p):a\in B\big\}
\]
are unital Banach algebras with the identities $p$ and $e-p$,
respectively.
\begin{lemma}\label{le:corners}
Let $B$ be a Banach algebra with identity $e$ and let $p\ne e$ be
an idempotent element of $B$. Suppose $A$ is a closed subalgebra
of $B$ that contains $e$ and $p$. If $p$ commutes with any element
of $A$, then
\begin{enumerate}
\item[{\rm (a)}]
an element $a\in A$ is invertible in the algebra $B$ if and only
if the element $pap$ is invertible in the algebra $pBp$ and the
element $(e-p)a(e-p)$ is invertible in the algebra $(e-p)B(e-p)$;

\item[{\rm (b)}]
the algebra $A$ is inverse closed in the algebra $B$ if and only
if the algebra $pAp$ is inverse closed in the algebra $pBp$ and the algebra
$(e-p)A(e-p)$ is inverse closed in the algebra $(e-p)B(e-p)$.
\end{enumerate}
\end{lemma}
The proof of this lemma is straightforward and therefore it is omitted.

The above results allow us to split the algebras
$\cA_{\infty,\eta}^\cJ$ and $\cL_{\infty,\eta}^\cJ$ into pairs of
simpler algebras $\cA_\eta^\pm$ and $\cL_\eta^\pm$. The algebras
$\cA_\eta^-$ and $\cA_\eta^+$ have a nice algebraic structure:
they are generated by two idempotents and the identity.
\begin{lemma}\label{le:2-proj}
Let $\eta\in M_\infty(SO)$ and
\[
\begin{array}{ll}
\cA_\eta^-:=\Phi_{\infty,\eta}^\cJ(\bX_-)\cA_{\infty,\eta}^\cJ\Phi_{\infty,\eta}^\cJ(\bX_-),
&
\cA_\eta^+:=\Phi_{\infty,\eta}^\cJ(\bX_+)\cA_{\infty,\eta}^\cJ\Phi_{\infty,\eta}^\cJ(\bX_+),
\\[2mm]
\cL_\eta^-:=\Phi_{\infty,\eta}^\cJ(\bX_-)\cL_{\infty,\eta}^\cJ\Phi_{\infty,\eta}^\cJ(\bX_-),
&
\cL_\eta^+:=\Phi_{\infty,\eta}^\cJ(\bX_+)\cL_{\infty,\eta}^\cJ\Phi_{\infty,\eta}^\cJ(\bX_+).
\end{array}
\]
\begin{enumerate}
\item[{\rm (a)}]
The elements
\[
p_-:=\Phi_{\infty,\eta}^\cJ(\bP\bX_-), \quad
r_-:=\Phi_{\infty,\eta}^\cJ(\bW_-\bX_-)
\]
are idempotents in the algebra $\cL_\eta^-$ and the algebra
$\cA_\eta^-$ is the smallest closed subalgebra of $\cL_\eta^-$
that contains $p_-,r_-$, and the identity
$e_-:=\Phi_{\infty,\eta}^\cJ(\bX_-)$.

\item[{\rm (b)}]
The elements
\[
p_+:=\Phi_{\infty,\eta}^\cJ(\bP\bX_+), \quad
r_+:=\Phi_{\infty,\eta}^\cJ(\bW_+\bX_+)
\]
are idempotents in the algebra $\cL_\eta^+$ and the algebra
$\cA_\eta^+$ is the smallest closed subalgebra of $\cL_\eta^+$
that contains $p_+,r_+$, and the identity
$e_+:=\Phi_{\infty,\eta}^\cJ(\bX_+)$.
\end{enumerate}
\end{lemma}
\begin{proof}
(a) Since $\Phi_{\infty,\eta}^\cJ(\bX_-)$ commutes with any element
of $\cA_{\infty,\eta}^\cJ$ by Lemma~\ref{le:commuting-idempotent} and
$\bX_-^2=\bX_-$, we have
\begin{eqnarray}
p_- &=&
\Phi_{\infty,\eta}^\cJ(\bP\bX_-)
=
\Phi_{\infty,\eta}^\cJ(\bX_-\bP\bX_-)
\nonumber
\\
&=&
\Phi_{\infty,\eta}^\cJ(\bX_-)
\Phi_{\infty,\eta}^\cJ(\bP)
\Phi_{\infty,\eta}^\cJ(\bX_-)
\in\cA_\eta^-
\label{eq:2-proj-1}
\end{eqnarray}
and taking into account that $\bP^2=\bP$, we have
\[
\begin{split}
p_-^2
&=
\Phi_{\infty,\eta}^\cJ(\bP\bX_-)\Phi_{\infty,\eta}^\cJ(\bP\bX_-)
=
\Phi_{\infty,\eta}^\cJ(\bP^2)\Phi_{\infty,\eta}^\cJ(\bX_-^2)
\\
&=
\Phi_{\infty,\eta}^\cJ(\bP)\Phi_{\infty,\eta}^\cJ(\bX_-)
=p_-.
\end{split}
\]
Analogously, taking into account that $\bW_-^2=\bW_-$, one can
verify that $r_-$ belongs to $\cA_\eta^-$ and $r_-^2=r_-$.

It remains to show that $\cA_\eta^-$ is generated by
$p_-$, $r_-$, and $e_-$. In view of \eqref{eq:2-proj-1} it is
sufficient to show that $\Phi_{\infty,\eta}^\cJ(\bX_-\bA\bX_-)$ and
$\Phi_{\infty,\eta}^\cJ(\bX_-\bB\bX_-)$, where $\bA=(aI)$ with $a\in
PC$ and $\bB=(W^0(b))$ with $b\in[PC_p,SO_p]$, are represented as
linear combinations of $p_-$, $r_-$, and $e_-$.

From Lemmas~\ref{le:case-3-repr} and \ref{le:commuting-idempotent} it follows that
\begin{eqnarray}
&&
\Phi_{\infty,\eta}^\cJ(\bX_-\bA\bX_-)
\nonumber
\\
&&=
\Phi_{\infty,\eta}^\cJ(\bX_-)
\big[a(-\infty)\Phi_{\infty,\eta}^\cJ(\bX_-)
+a(+\infty)\Phi_{\infty,\eta}^\cJ(\bX_+)\big]
\Phi_{\infty,\eta}^\cJ(\bX_-)
\nonumber
\\
&&=
a(-\infty)
\Phi_{\infty,\eta}^\cJ(\bX_-)\Phi_{\infty,\eta}^\cJ(\bX_-)\Phi_{\infty,\eta}^\cJ(\bX_-)
\nonumber
\\
&&\quad
+a(+\infty)
\Phi_{\infty,\eta}^\cJ(\bX_-)\Phi_{\infty,\eta}^\cJ(\bX_+)\Phi_{\infty,\eta}^\cJ(\bX_-)
\nonumber
\\
&&=
a(-\infty)e_-
\label{eq:2-proj-2}
\end{eqnarray}
and
\begin{eqnarray}
&&
\Phi_{\infty,\eta}^\cJ(\bX_-\bB\bX_-)
\nonumber
\\
&&=
\Phi_{\infty,\eta}^\cJ(\bX_-)
\big[b_\eta(-\infty)\Phi_{\infty,\eta}^\cJ(\bW_-)
+b_\eta(+\infty)\Phi_{\infty,\eta}^\cJ(\bW_+)\big]
\Phi_{\infty,\eta}^\cJ(\bX_-)
\nonumber
\\
&&=
b_\eta(-\infty)\Phi_{\infty,\eta}^\cJ(\bX_-)\Phi_{\infty,\eta}^\cJ(\bW_-)\Phi_{\infty,\eta}^\cJ(\bX_-)
\nonumber
\\
&&\quad+
b_\eta(+\infty)\Phi_{\infty,\eta}^\cJ(\bX_-)
[\Phi_{\infty,\eta}^\cJ(\bI)-\Phi_{\infty,\eta}^\cJ(\bW_-)]
\Phi_{\infty,\eta}^\cJ(\bX_-)
\nonumber
\\
&&=
b_\eta(-\infty)r_-+b_\eta(+\infty)(e_--r_-),
\label{eq:2-proj-3}
\end{eqnarray}
which finishes the proof of part (a).

\medskip
(b) The proof of part (b) is similar.  We only note that
\begin{eqnarray}
\Phi_{\infty,\eta}^\cJ(\bX_+\bP\bX_+) &=& p_+,
\label{eq:2-proj-4}
\\
\Phi_{\infty,\eta}^\cJ(\bX_+\bA\bX_+) &=& a(+\infty)e_+,
\label{eq:2-proj-5}
\\
\Phi_{\infty,\eta}^\cJ(\bX_+\bB\bX_+) &=&
b_\eta(-\infty)(e_+-r_+)+b_\eta(+\infty)r_+
\label{eq:2-proj-6}
\end{eqnarray}
for further references.
\end{proof}
\subsection{The two idempotents theorem}
The algebras $\cA_\eta^-$ and $\cA_\eta^+$ introduced in the
previous subsection are generated by two idempotent elements and
the identity. Invertibility of elements of such algebras can be
described with the aid of the so-called two idempotents theorem
presented below.
\begin{theorem}[{\cite[Theorem~8.7]{BK97}}]\label{th:two-idempotents}
Let $B$ be a Banach algebra with identity $e$ and let $p$ and $r$
be idempotents in $B$. Let further $A$ stand for the smallest
closed subalgebra of $B$ containing $e$, $p$, and $r$. Put
\begin{equation}\label{eq:element-X}
X:=prp+(e-p)(e-r)(e-p)
\end{equation}
and suppose the points $0$ and $1$ are cluster points of
$\operatorname{sp}_B(X)$. For $x\in\C$, define the map
$\sigma_x:\{e,p,r\}\to\C^{2\times 2}$ by
\begin{eqnarray}
&&
\sigma_x(e)=\left[\begin{array}{cc}
1 & 0\\ 0 & 1
\end{array}\right],
\quad
\sigma_x(p)=\left[\begin{array}{cc}
1 & 0\\ 0 & 0
\end{array}\right],
\label{eq:idempotents-maps}
\\
&&
\sigma_x(r)=\left[
\begin{array}{cc} x & \sqrt{x(1-x)} \\ \sqrt{x(1-x)} & 1-x\end{array}
\right],
\label{eq:idempotents-maps*}
\end{eqnarray}
where $\sqrt{x(1-x)}$ denotes any complex number such that its
square is equal to $x(1-x)$.
\begin{enumerate}
\item[{\rm(a)}]
For each $x\in\operatorname{sp}_B(X)$ the map $\sigma_x$ extends to
a Banach algebra homomorphism $\sigma_x$ of $A$ onto $\C^{2\times 2}$.

\item[{\rm(b)}]
An element $a\in A$ is invertible in $B$ if and only if
$\sigma_x(a)$ is invertible in $\C^{2\times 2}$ for every
$x\in\operatorname{sp}_B(X)$.

\item[{\rm(c)}]
An element $a\in A$ is invertible in $A$ if and only if
$\sigma_x(a)$ is invertible in $\C^{2\times 2}$ for every
$x\in\operatorname{sp}_A(X)$.
\end{enumerate}
\end{theorem}
\subsection{Spectra of two elements of the algebra \boldmath{$\cL_{\infty,\eta}^\cJ$}}
\label{subsection:spectra-pre-X}
To apply the two idempotents theorem to the algebras
$\cA_\eta^\pm$ and $\cL_\eta^\pm$, we need
to find the spectrum of the canonical element given by
\eqref{eq:element-X} in these algebras. This requires the following
auxiliary result.
\begin{lemma}\label{le:spectra-pre-X}
Let $\eta\in M_\infty(SO)$ and $\fL_p$ be
the lentiform domain defined by \eqref{eq:lens}. Then
\begin{eqnarray}
\operatorname{sp}_{\cL_{\infty,\eta}^\cJ} \left(\Phi_{\infty,\eta}^\cJ
(\bP\bX_-\bW_-\bX_-\bP+\bQ\bX_-\bW_+\bX_-\bQ)\right)
&=& \fL_p,
\label{eq:spectra-pre-X-1}
\\
\operatorname{sp}_{\cL_{\infty,\eta}^\cJ} \left(\Phi_{\infty,\eta}^\cJ
(\bP\bX_+\bW_+\bX_+\bP+\bQ\bX_+\bW_-\bX_+\bQ)\right)
&=& \fL_p.
\label{eq:spectra-pre-X-2}
\end{eqnarray}
\end{lemma}
\begin{proof}
Let us prove equality \eqref{eq:spectra-pre-X-1}. It is easy to
see that
\[
\bP\bX_-\bW_-\bX_-\bP+\bQ\bX_-\bW_+\bX_-\bQ-\lambda \bI\in\cH\cap\cL,
\]
where the algebra $\cH$ is defined in Subsection~\ref{subsection:homo}.
By Lemma~\ref{le:homogeneous}, the coset
\[
\bP\bX_-\bW_-\bX_-\bP+\bQ\bX_-\bW_+\bX_-\bQ-\lambda\bI+\cG
\]
is invertible in the algebra $\cL/\cG$ if and only if the operator
\[
\begin{split}
&
P_1\chi_-W^0(\chi_-)\chi_-P_1+Q_1\chi_-W^0(\chi_+)\chi_-Q_1-\lambda I
\\
&=
\chi_{(-1,0)}P_\R\chi_{(-1,0)}I+\chi_{(-\infty,-1)}Q_\R\chi_{(-\infty,-1)}I-\lambda I
\end{split}
\]
is invertible on $L^p(\R)$. Hence, taking into account Corollary~\ref{co:spectra-SIO}(a), we obtain
\[
\begin{split}
&
\operatorname{sp}_{\cL/\cG}
(\bP\bX_-\bW_-\bX_-\bP+\bQ\bX_-\bW_+\bX_-\bQ+\cG)
\\
&=
\operatorname{sp}_{\cB}
\left(\chi_{(-1,0)}P_\R\chi_{(-1,0)}I+\chi_{(-\infty,-1)}Q_\R\chi_{(-\infty,-1)}I\right)
=\fL_p.
\end{split}
\]
If $\lambda\notin\fL_p$, then there exists a sequence $\bA\in\cL$ such that
\[
\begin{split}
&
(\bP\bX_-\bW_-\bX_-\bP+\bQ\bX_-\bW_+\bX_-\bQ-\lambda\bI+\cG)(\bA+\cG)=\bI+\cG,
\\
&
(\bA+\cG)(\bP\bX_-\bW_-\bX_-\bP+\bQ\bX_-\bW_+\bX_-\bQ-\lambda\bI+\cG)=\bI+\cG.
\end{split}
\]
Obviously, the same equalities are true with $\cG$ replaced by the
larger ideal $\cJ$. Applying the homomorphism
$\Phi_{\infty,\eta}^\cJ:\cL^\cJ\to\cL_{\infty,\eta}^\cJ$ to those
equalities, we obtain
\[
\begin{split}
&
\Phi_{\infty,\eta}^\cJ(\bP\bX_-\bW_-\bX_-\bP+\bQ\bX_-\bW_+\bX_-\bQ-\lambda\bI)
\Phi_{\infty,\eta}^\cJ(\bA)
=\Phi_{\infty,\eta}^\cJ(\bI),
\\
&
\Phi_{\infty,\eta}^\cJ(\bA)
\Phi_{\infty,\eta}^\cJ(\bP\bX_-\bW_-\bX_-\bP+\bQ\bX_-\bW_+\bX_-\bQ-\lambda\bI)
=\Phi_{\infty,\eta}^\cJ(\bI).
\end{split}
\]
Then
$\lambda\notin\operatorname{sp}_{\cL_{\infty,\eta}^\cJ}
\left(\Phi_{\infty,\eta}^\cJ(\bP\bX_-\bW_-\bX_-\bP+\bQ\bX_-\bW_+\bX_-\bQ)\right)$. Thus
\begin{equation}\label{eq:spectra-pre-X-3}
\operatorname{sp}_{\cL_{\infty,\eta}^\cJ}
\left(\Phi_{\infty,\eta}^\cJ(\bP\bX_-\bW_-\bX_-\bP+\bQ\bX_-\bW_+\bX_-\bQ)\right)
\subset\fL_p.
\end{equation}

Assume now that
\[
\lambda\notin\operatorname{sp}_{\cL_{\infty,\eta}^\cJ}
\left(\Phi_{\infty,\eta}^\cJ(\bP\bX_-\bW_-\bX_-\bP+\bQ\bX_-\bW_+\bX_-\bQ)\right).
\]
Then there exists a sequence $\bB\in\cL$ such that
\[
\begin{split}
&
\bB(\bP\bX_-\bW_-\bX_-\bP+\bQ\bX_-\bW_+\bX_-\bQ-\lambda\bI)-\bI\in\cI_{\infty,\eta}^\cJ,
\\
&
(\bP\bX_-\bW_-\bX_-\bP+\bQ\bX_-\bW_+\bX_-\bQ-\lambda\bI)\bB-\bI\in\cI_{\infty,\eta}^\cJ,
\end{split}
\]
where $\cI_{\infty,\eta}^\cJ$ was defined in Subsection~\ref{subsec:localization-our}. From that definition
and the above relations it follows that without loss of generality we
may assume that there exist finite sums
\[
\begin{split}
&
\bJ_1+\sum_{i=1}^n\bC_i(f_iI)+\sum_{j=1}^m\bD_j(W^0(g_j))\in\cI_{\infty,\eta}^\cJ,
\\
&
\bJ_2+\sum_{i=1}^k(h_iI)\bE_i+\sum_{j=1}^l(W^0(v_j))\bF_j\in\cI_{\infty,\eta}^\cJ
\end{split}
\]
with $\bC_i,\bD_j,\bE_i,\bF_j\in\cL$ and
\[
f_i,h_i\in C(\dot{\R}),\quad f_i(\infty)=h_i(\infty)=0,
\quad
g_j,v_j\in SO_p,\quad g_j(\eta)=v_j(\eta)=0
\]
for all $i,j$ and elements $\bJ_1,\bJ_2\in\cJ$ such that
\[
\bB(\bP\bX_-\bW_-\bX_-\bP+\bQ\bX_-\bW_+\bX_-\bQ-\lambda\bI)-\bI
= \bJ_1+\sum_{i=1}^n\bC_i(f_iI)+\sum_{j=1}^m\bD_j(W^0(g_j)),
\]
\[
(\bP\bX_-\bW_-\bX_-\bP+\bQ\bX_-\bW_+\bX_-\bQ-\lambda\bI)\bB-\bI
=\bJ_2+\sum_{i=1}^k(h_iI)\bE_i+\sum_{j=1}^l(W^0(v_j))\bF_j.
\]
Applying the homomorphism $\fW_{-1}$ to these equalities, taking into
account Proposition~\ref{pr:A-in-F} and that
$\fW_{-1}(\bJ_1)=K_1\in\cK$, $\fW_{-1}(\bJ_2)=K_2\in\cK$, we obtain
\begin{eqnarray}
\fW_{-1}(\bB)H_\lambda-I &=& K_1+\sum_{j=1}^m\fW_{-1}(\bD_j)W^0(g_j),
\label{eq:spectra-pre-X-4}
\\
H_\lambda\fW_{-1}(\bB)-I &=& K_2+\sum_{j=1}^l W^0(v_j)\fW_{-1}(\bF_j),
\label{eq:spectra-pre-X-5}
\end{eqnarray}
where
\[
H_\lambda
:=
\chi_+W^0(\chi_-)\chi_+I+\chi_-W^0(\chi_+)\chi_-I-\lambda I
=
\chi_+P_\R\chi_+I+\chi_-Q_\R\chi_-I-\lambda I.
\]

By Lemma~\ref{le:family-convergence}, for $g_1,\dots,g_m$ there exists a
sequence $(\tau_n)_{n=1}^\infty\subset\R_+$ such that $\tau_n\to+\infty$ and
\begin{equation}\label{eq:spectra-pre-X-6}
\operatornamewithlimits{s-lim}_{n\to\infty}
Z_{\tau_n}W^0(g_j)Z_{\tau_n}^{-1}=0
\quad\mbox{on}\quad L^p(\R)
\end{equation}
for all $j=1,\dots,m$. From Lemma~\ref{le:SOq} it follows that
$\overline{v_1},\dots,\overline{v_l}\in SO_q$. So, applying Lemma~\ref{le:family-convergence}
to the space $L^q(\R)$, we can guarantee that there exists a sequence
$(\tau_n^*)_{n=1}^\infty\subset\R_+$ such that $\tau_n^*\to+\infty$ and
\begin{equation}\label{eq:spectra-pre-X-7}
\operatornamewithlimits{s-lim}_{n\to\infty}
Z_{\tau_n^*}W^0(\overline{v_j})Z_{\tau_n^*}^{-1}=0
\quad\mbox{on}\quad L^q(\R)
\end{equation}
for all $j=1,\dots,l$.

Since $K_1,K_2$ are compact and $(Z_{\tau_n}^{\pm 1})_{n=1}^\infty$ converge
weakly to zero on $L^p(\R)$ and $(Z_{\tau_n^*}^{\pm 1})_{n=1}^\infty$ converge
weakly to zero on $L^q(\R)$ as $n\to\infty$, we
conclude that
\begin{eqnarray}
&&
\operatornamewithlimits{s-lim}_{n\to\infty}
Z_{\tau_n}K_1Z_{\tau_n}^{-1}=0
\quad\mbox{on}\quad L^p(\R),
\label{eq:spectra-pre-X-8}
\\
&&
\operatornamewithlimits{s-lim}_{n\to\infty}
Z_{\tau_n^*}K_2^*Z_{\tau_n^*}^{-1}=0
\quad\mbox{on}\quad L^q(\R).
\label{eq:spectra-pre-X-9}
\end{eqnarray}
Multiplying \eqref{eq:spectra-pre-X-4} from the left by
$Z_{\tau_n}$ and from the right by $Z_{\tau_n}^{-1}$ and taking into account that
the operator $H_\lambda$ is homogeneous, we obtain
\begin{eqnarray}
&&
(Z_{\tau_n}\fW_{-1}(\bB)Z_{\tau_n}^{-1})H_\lambda-I
\nonumber
\\
&&=
Z_{\tau_n}K_1Z_{\tau_n}^{-1}
+
\sum_{j=1}^m
(Z_{\tau_n}\fW_{-1}(\bD_j)Z_{\tau_n}^{-1})
(Z_{\tau_n}W^0(g_j)Z_{\tau_n}^{-1})=:A_n.
\label{eq:spectra-pre-X-10}
\end{eqnarray}
Since
$\|Z_{\tau_n}\fW_{-1}(\bD_j)Z_{\tau_n}^{-1}\|_\cB\le\|\fW_{-1}(\bD_j)\|_\cB$,
from \eqref{eq:spectra-pre-X-6} and \eqref{eq:spectra-pre-X-8} we
conclude that the sequence $(A_n)_{n=1}^\infty$ converges strongly
to zero as $n\to\infty$. For $u\in L^p(\R)$ from
\eqref{eq:spectra-pre-X-10} it follows that
\[
\|u\|_p \le
\|Z_{\tau_n}\fW_{-1}(\bB)Z_{\tau_n}^{-1}\|_\cB \|H_\lambda u\|_p+\|A_nu\|_p
\le
\|\fW_{-1}(\bB)\|_\cB\|H_\lambda u\|_p+\|A_nu\|_p.
\]
Passing to the limit as $n\to\infty$, we obtain
\[
\|u\|_p\le \|\fW_{-1}(\bB)\|_\cB\|H_\lambda u\|_p
\]
for all $u\in L^p(\R)$. This means that the kernel of $H_\lambda$ is trivial
and the range of $H_\lambda$ is closed.

Analogously, from \eqref{eq:spectra-pre-X-5} it follows that
\begin{eqnarray*}
&&
(Z_{\tau_n^*}W_{-1}(\bB)^*Z_{\tau_n^*}^{-1})H_\lambda^*-I
\\
&&=
Z_{\tau_n^*}K_2^*Z_{\tau_n^*}^{-1}+
\sum_{j=1}^l(Z_{\tau_n^*}\fW_{-1}(\bF_j)Z_{\tau_n^*}^{-1})(Z_{\tau_n^*}W^0(\overline{v_j})Z_{\tau_n^*}^{-1})=:B_n.
\end{eqnarray*}
As above, taking into account \eqref{eq:spectra-pre-X-7}
and \eqref{eq:spectra-pre-X-9}, we see that $(B_n)_{n=1}^\infty$
converges strongly to zero on $L^q(\R)$ and therefore
\[
\|w\|_q\le\|\fW_{-1}(\bB)\|_\cB\|H_\lambda^*w\|_q
\]
for every $w\in L^q(\R)$.
Hence the kernel of $H_\lambda^*$ is trivial, too. Thus $H_\lambda$ is invertible
on $L^p(\R)$. From Corollary~\ref{co:spectra-SIO}(c) we conclude that
$\lambda\notin\fL_p$. Thus
\begin{equation}\label{eq:spectra-pre-X-11}
\fL_p\subset \operatorname{sp}_{\cL_{\infty,\eta}^\cJ}
\left(\Phi_{\infty,\eta}^\cJ(\bP\bX_-\bW_-\bX_-\bP+\bQ\bX_-\bW_+\bX_-\bQ)\right).
\end{equation}
Combining \eqref{eq:spectra-pre-X-3} and
\eqref{eq:spectra-pre-X-11}, we arrive at
\eqref{eq:spectra-pre-X-1}.

To prove equality \eqref{eq:spectra-pre-X-2}, we argue similarly,
applying Corollary~\ref{co:spectra-SIO}(b) instead of Corollary~\ref{co:spectra-SIO}(a)
and the homomorphism $\fW_1$ instead of $\fW_{-1}$.
\end{proof}
\subsection{Spectra of the canonical elements in the algebras
\boldmath{$\cA_\eta^\pm$} and \boldmath{$\cL_\eta^\pm$}}
Now we are ready to calculate the spectra of the canonical elements
required in the two idempotents theorem.
\begin{lemma}\label{le:spectra-X}
Let $\eta\in M_\infty(SO)$ and $\fL_p$ be the lentiform domain defined by \eqref{eq:lens}.
\begin{enumerate}
\item[{\rm (a)}]
If
\[
X_-:=p_-r_-p_-+(e_--p_-)(e_--r_-)(e_--p_-),
\]
then $\operatorname{sp}_{\cA_\eta^-}(X_-) =
\operatorname{sp}_{\cL_\eta^-}(X_-) =\fL_p$.

\item[{\rm (b)}]
If
\[
X_+:=p_+r_+p_++(e_+-p_+)(e_+-r_+)(e_+-p_+),
\]
then $\operatorname{sp}_{\cA_\eta^+}(X_+) =
\operatorname{sp}_{\cL_\eta^+}(X_+) =\fL_p$.
\end{enumerate}
\end{lemma}
\begin{proof}
(a) Let $\lambda\in\C$. From Lemmas~\ref{le:corners}(a) and
\ref{le:commuting-idempotent} it follows that the element
$\Phi_{\infty,\eta}^\cJ(\bP\bX_-\bW_-\bX_-\bP+\bQ\bX_-\bW_+\bX_-\bQ-\lambda\bI)$
is invertible in $\cL_{\infty,\eta}^\cJ$ if and only if
\[
\Phi_{\infty,\eta}^\cJ(\bX_-)
\Phi_{\infty,\eta}^\cJ(\bP\bX_-\bW_-\bX_-\bP+\bQ\bX_-\bW_+\bX_-\bQ-\lambda\bI)
\Phi_{\infty,\eta}^\cJ(\bX_-)
=X_--\lambda e_-
\]
is invertible in $\cL_\eta^-$ and
\[
\Phi_{\infty,\eta}^\cJ(\bX_+)
\Phi_{\infty,\eta}^\cJ(\bP\bX_-\bW_-\bX_-\bP+\bQ\bX_-\bW_+\bX_-\bQ-\lambda\bI)
\Phi_{\infty,\eta}^\cJ(\bX_+)
=-\lambda e_+
\]
is invertible in $\cL_\eta^+$. Therefore
\[
\operatorname{sp}_{\cL_{\infty,\eta}^\cJ} \left(\Phi_{\infty,\eta}^\cJ
(\bP\bX_-\bW_-\bX_-\bP+\bQ\bX_-\bW_+\bX_-\bQ)\right)
=\operatorname{sp}_{\cL_\eta^-}(X_-)\cup\{0\}.
\]
In view of \eqref{eq:spectra-pre-X-1}, the above equality means that
$\operatorname{sp}_{\cL_\eta^-}(X_-)\cup\{0\}=\fL_p$.
But $0$ is not an isolated point of $\fL_p$. Due to the
compactness of the spectrum
$\operatorname{sp}_{\cL_\eta^-}(X_-)$ we conclude that
$0\in\operatorname{sp}_{\cL_\eta^-}(X_-)$. Hence
\begin{equation}\label{eq:spectra-X-1}
\operatorname{sp}_{\cL_\eta^-}(X_-)=\fL_p.
\end{equation}

Recall that the identity element in the unital algebras
$\cA_\eta^-\subset\cL_\eta^-$ is the same.
Since the lentiform domain $\fL_p$ does not separate the complex
plane, that is, $\fL_p$ and $\C\setminus\fL_p$ are connected sets,
from \cite[Corollary of Theorem 10.18]{Rudin91} and equality
\eqref{eq:spectra-X-1} it follows that
$\operatorname{sp}_{\cA_\eta^-}(X_-)=\fL_p$. Part (a) is
proved.

The proof of part (b) is analogous.
\end{proof}
\subsection{Invertibility in the local algebra \boldmath{$\cL_{\infty,\eta}^\cJ$}}
Now we are in a position to describe the local algebras $\cA_{\infty,\eta}^\cJ$
with the aid of the two idempotents theorem.
\begin{theorem}\label{th:case-3}
Let $\eta\in M_\infty(SO)$ and $\fL_p$ be
the lentiform domain defined by \eqref{eq:lens}. Suppose $x\in\fL_p$.
\begin{enumerate}
\item[{\rm (a)}]
The maps
\[
\begin{split}
\Sigma_x^-,\Sigma_x^+: & \big\{\Phi_{\infty,\eta}^\cJ(\bP)\big\} \cup
\big\{\Phi_{\infty,\eta}^\cJ\big[(aI)\big]:a\in PC\big\}
\\
&\cup
\big\{\Phi_{\infty,\eta}^\cJ\big[(W^0(b))\big]:b\in[PC_p,SO_p]\big\}
\to\C^{2\times 2}
\end{split}
\]
given by
\begin{eqnarray}
\hspace{-15mm}
&&
\Sigma_x^-(\Phi_{\infty,\eta}^\cJ(\bP)):=
\left[\begin{array}{cc} 1 & 0 \\ 0 & 0\end{array}\right],
\
\Sigma_x^-(\Phi_{\infty,\eta}^\cJ[(aI)]) :=
\left[\begin{array}{cc}
a(-\infty) & 0 \\ 0 & a(-\infty)\end{array}\right],
\label{eq:case-3-1}
\\[3mm]
\hspace{-15mm}
&&
\Sigma_x^-(\Phi_{\infty,\eta}^\cJ[(W^0(b)]):=
\nonumber
\\[2mm]
\hspace{-15mm}
&&
\left[\begin{array}{cc}
b_\eta(-\infty)x+b_\eta(+\infty)(1-x) &
[b_\eta(-\infty)-b_\eta(+\infty)]\sqrt{x(1-x)}
\\[2mm]
[b_\eta(-\infty)-b_\eta(+\infty)]\sqrt{x(1-x)} &
b_\eta(-\infty)(1-x)+b_\eta(+\infty)x
\end{array}\right]
\label{eq:case-3-2}
\end{eqnarray}
and
\begin{eqnarray}
\hspace{-15mm}
&&
\Sigma_x^+(\Phi_{\infty,\eta}^\cJ(\bP)):=
\left[\begin{array}{cc} 1 & 0 \\ 0 & 0\end{array}\right],
\
\Sigma_x^+(\Phi_{\infty,\eta}^\cJ[(aI)]) :=
\left[\begin{array}{cc} a(+\infty) & 0 \\ 0 & a(+\infty)\end{array}\right],
\label{eq:case-3-3}
\\[3mm]
\hspace{-15mm}
&&
\Sigma_x^+(\Phi_{\infty,\eta}^\cJ[(W^0(b)]):=
\nonumber
\\[2mm]
\hspace{-15mm}
&&
\left[\begin{array}{cc}
b_\eta(+\infty)x+b_\eta(-\infty)(1-x) &
[b_\eta(+\infty)-b_\eta(-\infty)]\sqrt{x(1-x)}
\\[2mm]
[b_\eta(+\infty)-b_\eta(-\infty)]\sqrt{x(1-x)} &
b_\eta(+\infty)(1-x)+b_\eta(-\infty)x
\end{array}\right],
\label{eq:case-3-4}
\end{eqnarray}
where $\sqrt{x(1-x)}$ denotes any complex number such that its
square is equal to $x(1-x)$, extend to Banach algebra
homomorphisms
\[
\Sigma_x^-:\cA_{\infty,\eta}^\cJ\to\C^{2\times 2},
\quad
\Sigma_x^+:\cA_{\infty,\eta}^\cJ\to\C^{2\times 2}.
\]

\item[{\rm (b)}]
Let $\bA\in\cA$. The coset $\Phi_{\infty,\eta}^\cJ(\bA)$ is
invertible in the local algebra $\cL_{\infty,\eta}^\cJ$ if and only
if
\[
\det\Sigma_x^-(\Phi_{\infty,\eta}^\cJ(\bA))\ne 0, \quad
\det\Sigma_x^+(\Phi_{\infty,\eta}^\cJ(\bA))\ne 0
\]
for all $x\in\fL_p$.

\item[{\rm (c)}]
The algebra $\cA_{\infty,\eta}^\cJ$ is inverse closed in the algebra
$\cL_{\infty,\eta}^\cJ$.
\end{enumerate}
\end{theorem}
\begin{proof}
From Lemma~\ref{le:2-proj} we obtain that the algebras
$\cA_\eta^-$ and $\cA_\eta^+$ are subject
to the two idempotents theorem. Let $x\in\C$. Define the maps
\[
\sigma_x^\pm:\{e_\pm,p_\pm,r_\pm\}\to\C^{2\times 2}
\]
by equalities
\eqref{eq:idempotents-maps}--\eqref{eq:idempotents-maps*} with $e_\pm,p_\pm,r_\pm$ and
$\sigma_x^\pm$ in place of $e,p,r$, and $\sigma_x$, respectively.
By Lemma~\ref{le:spectra-X}, the spectra
\[
\operatorname{sp}_{\cA_\eta^-}(X_-),\quad
\operatorname{sp}_{\cL_\eta^-}(X_-),\quad
\operatorname{sp}_{\cA_\eta^+}(X_+),\quad
\operatorname{sp}_{\cL_\eta^+}(X_+)
\]
all coincide and
are equal to the lentiform domain $\fL_p$ given by
\eqref{eq:lens}. Then, by Theorem~\ref{th:two-idempotents}(a), for
$x\in\fL_p$ the maps $\sigma_x^-$ and $\sigma_x^+$  extend to
Banach algebra homomorphisms
\[
\sigma_x^-:\cA_\eta^-\to\C^{2\times 2},
\quad
\sigma_x^+:\cA_\eta^+\to\C^{2\times 2}.
\]
Let $\bA\in\cA$. From Theorem~\ref{th:two-idempotents}(b),(c)  we
know that the element $\Phi_{\infty,\eta}^\cJ(\bX_-\bA\bX_-)$ (respectively,
$\Phi_{\infty,\eta}^\cJ(\bX_+\bA\bX_+)$) is invertible in both algebras
$\cA_\eta^-$ and $\cL_\eta^-$ (respectively, in
both $\cA_\eta^+$ and $\cL_\eta^+$) if and
only if the element $\sigma_x^-[\Phi_{\infty,\eta}^\cJ(\bX_-\bA\bX_-)]$ (respectively,
$\sigma_x^+[\Phi_{\infty,\eta}^\cJ(\bX_+\bA\bX_+)]$) is invertible in
$\C^{2\times 2}$. In particular, the algebra
$\cA_\eta^-$ is inverse closed in the algebra
$\cL_\eta^-$ and the algebra $\cA_\eta^+$
is inverse closed in the algebra $\cL_\eta^+$.

Further, for $x\in\fL_p$ the maps
\[
\begin{split}
& \Sigma_x^-:\cA_{\infty,\eta}^\cJ\to\C^{2\times 2}, \quad
\Sigma_x^-[\Phi_{\infty,\eta}^\cJ(\bA)]=
\sigma_x^-[\Phi_{\infty,\eta}^\cJ(\bX_-\bA\bX_-)],
\\
&
\Sigma_x^+:\cA_{\infty,\eta}^\cJ\to\C^{2\times 2}, \quad
\Sigma_x^+[\Phi_{\infty,\eta}^\cJ(\bA)]=
\sigma_x^+[\Phi_{\infty,\eta}^\cJ(\bX_+\bA\bX_+)]
\end{split}
\]
are Banach algebra homomorphisms. From \eqref{eq:idempotents-maps}--\eqref{eq:idempotents-maps*}
and \eqref{eq:2-proj-1}--\eqref{eq:2-proj-6} it follows that
$\Sigma_x^-$ and $\Sigma_x^+$ are defined on the generators of
$\cA_{\infty,\eta}^\cJ$ by formulas
\eqref{eq:case-3-1}--\eqref{eq:case-3-4}. This finishes the proof of
parts (a) and (b).

Taking into account that $\cA_\eta^-$ is inverse
closed in $\cL_\eta^-$ and $\cA_\eta^+$ is
inverse closed in $\cL_\eta^+$, we obtain that
$\cA_{\infty,\eta}^\cJ$ is inverse closed in $\cL_{\infty,\eta}^\cJ$ in
view of Lemma~\ref{le:corners}(b). Part (c) is proved.
\end{proof}
Let $R:\C\setminus\big((-\infty,0)\cup(1,+\infty)\big)\to\C$ be a
continuous branch of the function $x\mapsto\sqrt{x(1-x)}$. Since
$\fL_p\cap\big((-\infty,0)\cup(1,+\infty)\big)=\emptyset$, the
function $R$ is continuous on $\fL_p$.
\begin{corollary}\label{co:case-3}
Let $\eta\in M_\infty(SO)$ and $\fL_p$ be
the lentiform domain defined by \eqref{eq:lens}. The maps
\[
\fN_\eta^-,\fN_\eta^+:\{\bP\}\cup \big\{(aI):a\in
PC\big\}\cup\big\{(W^0(b)):b\in[PC_p,SO_p]\big\}\to
C(\fL_p,\C^{2\times 2})
\]
given for $x\in\fL_p$ by
\begin{eqnarray}
\hspace{-15mm}
&& \big(\fN_\eta^-(\bP)\big)(x):=
\left[\begin{array}{cc} 1 & 0 \\ 0 & 0\end{array}\right], \quad
\big(\fN_\eta^-[(aI)]\big)(x) := \left[\begin{array}{cc}
a(-\infty) & 0 \\ 0 & a(-\infty)\end{array}\right],
\label{eq:case-3-5}
\\[3mm]
\hspace{-15mm}
&&
\big(\fN_\eta^-[(W^0(b)]\big)(x):=
\nonumber
\\[2mm]
\hspace{-15mm}
&&
\left[\begin{array}{cc}
b_\eta(-\infty)x+b_\eta(+\infty)(1-x) &
[b_\eta(-\infty)-b_\eta(+\infty)]R(x)
\\[2mm]
[b_\eta(-\infty)-b_\eta(+\infty)]R(x) &
b_\eta(-\infty)(1-x)+b_\eta(+\infty)x
\end{array}\right]
\label{eq:case-3-6}
\end{eqnarray}
and
\begin{eqnarray}
\hspace{-15mm}
&&
\big(\fN_\eta^+(\bP)\big)(x):=
\left[\begin{array}{cc} 1 & 0 \\ 0 & 0\end{array}\right], \quad
\big(\fN_\eta^+[(aI)]\big)(x) := \left[\begin{array}{cc}
a(+\infty) & 0 \\ 0 & a(+\infty)\end{array}\right],
\label{eq:case-3-7}
\\[3mm]
\hspace{-15mm}
&&
\big(\fN_\eta^+[(W^0(b)]\big)(x):=
\nonumber
\\[2mm]
\hspace{-15mm}
&&
\left[\begin{array}{cc}
b_\eta(+\infty)x+b_\eta(-\infty)(1-x) &
[b_\eta(+\infty)-b_\eta(-\infty)]R(x)
\\[2mm]
[b_\eta(+\infty)-b_\eta(-\infty)]R(x) &
b_\eta(+\infty)(1-x)+b_\eta(-\infty)x
\end{array}\right],
\label{eq:case-3-8}
\end{eqnarray}
extend to Banach algebra homomorphisms
\[
\fN_\eta^-:\cA\to C(\fL_p,\C^{2\times 2}), \quad \fN_\eta^+:\cA\to
C(\fL_p,\C^{2\times 2}).
\]
\end{corollary}
\begin{proof}
It is clear that $\fN_\eta^-$ (resp. $\fN_\eta^+$) is the
composition of the canonical Banach algebra homomorphism
$\Phi_{\infty,\eta}^\cJ|_\cA:\cA\to\cA_{\infty,\eta}^\cJ$ and the
Banach algebra homomorphism $\Sigma_x^-$ (resp. $\Sigma_x^+$)
defined in Theorem~\ref{th:case-3}(a). Note that the function $R$
was chosen so that $\Sigma_x^\pm(\Phi_{\infty,\eta}^\cJ(\bA))$ are
continuous in $x\in\fL_p$ for $\bA\in\cA$.
\end{proof}
\section{Main result and some corollaries}\label{section:main}
\subsection{Main result}
Combining the results of
Sections~\ref{section:localization}--\ref{section:identification-3},
we arrive at the main result of the paper.
\begin{theorem}\label{th:main}
Let $\fL_p$ be the lentiform domain defined by \eqref{eq:lens} and
$\cA$ be the algebra from Definition~\ref{def:alg-A}. A sequence
$\bA\in\cA$ is stable if and only if the following three
conditions are fulfilled:
\begin{enumerate}
\item[{\rm(a)}]
the operators $\fW_{-1}(\bA)$, $\fW_0(\bA)$, and $\fW_1(\bA)$ are
invertible in $\cB$;

\item[{\rm(b)}]
the operators $\fH_\eta(\bA)$ are invertible in $\cB$ for all
$\eta\in\R$;

\item[{\rm(c)}] for every $\eta\in M_\infty(SO)$ and every
$x\in\fL_p$,
\[
\det\big(\fN_\eta^-(\bA)\big)(x)\ne 0, \quad
\det\big(\fN_\eta^+(\bA)\big)(x)\ne 0,
\]
where
$\fN_\eta^\pm(\bA)$
are images of the
homomorphisms defined in Corollary~{\rm\ref{co:case-3}.}
\end{enumerate}
\end{theorem}
\begin{proof}
\textit{Necessity.} If $\bA\in\cA$ is stable, then condition (b) holds
due to Theorem~\ref{th:nec-stability-Heta}. Further, in view of
Theorem~\ref{th:localization}, condition (a) is fulfilled and the
coset $\Phi_{\xi,\eta}^\cJ(\bA)$ is invertible in the local
algebra $\cL_{\xi,\eta}^\cJ$ for every pair $(\xi,\eta)\in\Omega$,
where $\Omega$ is given by \eqref{eq:def-Omega}. If
$(\xi,\eta)\in\{\infty\}\times M_\infty(SO)$, then applying
Theorem~\ref{th:case-3}(b) we obtain for all $x\in\fL_p$,
\[
\det\Sigma_x^-(\Phi_{\infty,\eta}^\cJ(\bA))\ne 0,
\quad
\det\Sigma_x^+(\Phi_{\infty,\eta}^\cJ(\bA))\ne 0.
\]
Since
$(\fN_\eta^-(\bA))(x)=\Sigma_x^-(\Phi_{\infty,\eta}^\cJ(\bA))$ and
$(\fN_\eta^+(\bA))(x)=\Sigma_x^+(\Phi_{\infty,\eta}^\cJ(\bA))$, this
immediately implies condition (c). The necessity portion is
proved.

\textit{Sufficiency.} If condition (a) is fulfilled, then
from Theorem~\ref{th:case-1} it follows that 
$\Phi_{\xi,\eta}^\cJ(\bA)$ is invertible in the local algebra
$\cL_{\xi,\eta}^\cJ$ with $(\xi,\eta)\in\R\times M_\infty(SO)$.
From condition (b) and Theorem~\ref{th:case-2} we infer that the coset
$\Phi_{\infty,\eta}^\cJ(\bA)$ is invertible in the local algebra
$\cL_{\infty,\eta}^\cJ$ with $\eta\in\R$.
From condition (c) in view of Theorem~\ref{th:case-3}(b) and
Corollary~\ref{co:case-3} it follows that the coset
$\Phi_{\infty,\eta}^\cJ(\bA)$ is invertible in the local algebra
$\cL_{\infty,\eta}^\cJ$ with $\eta\in M_\infty(SO)$. Thus the coset
$\Phi_{\xi,\eta}^\cJ(\bA)$ is
invertible in the local algebra $\cL_{\xi,\eta}^\cJ$ for every
$(\xi,\eta)\in\Omega$, where $\Omega$ is given by
\eqref{eq:def-Omega}. Combining this fact with condition (a) we
see that $\bA$ is stable by Theorem~\ref{th:localization}.
\end{proof}
\subsection{Applicability of the finite section method}
Theorem~\ref{th:main} can be specified for the finite section
method as follows.
\begin{theorem}\label{th:finite-section-method}
Suppose $\fL_p$ is a lentiform domain defined by \eqref{eq:lens}.
Let $A$ be an operator in the smallest closed subalgebra of the
algebra $\cB$ that contains the operators $aI$ with $a\in PC$ and
$W^0(b)$ with $b\in [PC_p,SO_p]$. Then the finite section method
\eqref{eq:FSM} applies to the operator $A$ if and only if the
following three conditions hold with $\bA=(A)$:
\begin{enumerate}
\item[{\rm(a)}] the operators
\[
\chi_+\fW_{-1}(\bA)\chi_+I+\chi_-I,
\quad
\chi_-\fW_1(\bA)\chi_-I+\chi_+I,
\]
and $A$ are invertible in $\cB$;

\item[{\rm (b)}] the operators $P_1\fH_\eta(\bA)P_1+Q_1$ are invertible in $\cB$ for
all $\eta\in\R$;

\item[{\rm (c)}] for every $\eta\in M_\infty(SO)$ and every
$x\in\fL_p$,
\[
[\fN_\eta^-(\bA)]_{11}(x)\ne 0, \quad [\fN_\eta^+(\bA)]_{11}(x)\ne
0,
\]
where $[f]_{11}$ denotes the $(1,1)$-entry of a $2\times 2$ matrix
function $f$.
\end{enumerate}
\end{theorem}
\begin{proof}
From Proposition~\ref{pr:A-in-F}(a) it follows that
\[
\begin{split}
\fW_{-1}(\bP\bA\bP+\bQ) &=\chi_+\fW_{-1}(\bA)\chi_+I+\chi_-I,
\\
\fW_0(\bP\bA\bP+\bQ) &=A,
\\
\fW_1(\bP\bA\bP+\bQ) &=\chi_-\fW_1(\bA)\chi_-I+\chi_+I.
\end{split}
\]
Hence, condition (a) of Theorem~\ref{th:main} for the sequence
$\bP\bA\bP+\bQ$ is just condition (a) of the present theorem.

From Proposition~\ref{pr:A-in-Heta}(a) we obtain for every $\eta\in\R$,
\[
\fH_\eta(\bP\bA\bP+\bQ)=P_1\fH_\eta(\bA)P_1+Q_1.
\]
This means that condition (b) of Theorem~\ref{th:main} specifies
to condition (b) of the present theorem for the sequence
$\bP\bA\bP+\bQ$.

Finally from \eqref{eq:case-3-5} and \eqref{eq:case-3-7} we get
for every $\eta\in M_\infty(SO)$ and $x\in\fL_p$,
\begin{eqnarray}
\det\big(\fN_\eta^-(\bP\bA\bP+\bQ)\big)(x)
&=&\det\left(
\left[\begin{array}{cc} 1 & 0\\ 0 & 0\end{array}\right]
\fN_\eta^-(\bA)(x)
\left[\begin{array}{cc} 1 & 0\\ 0 & 0\end{array}\right]
+
\left[\begin{array}{cc} 0 & 0\\ 0 & 1\end{array}\right]
\right)
\nonumber
\\[2mm]
&=&\det\left[\begin{array}{cc}
[\fN_\eta^-(\bA)]_{11}(x) & 0 \\ 0 & 1
\end{array}\right]
=[\fN_\eta^-(\bA)]_{11}(x)
\label{eq:finite-section-method-1}
\end{eqnarray}
and analogously
\begin{equation}\label{eq:finite-section-method-2}
\det\big(\fN_\eta^+(\bP\bA\bP+\bQ)\big)(x)=[\fN_\eta^+(\bA)]_{11}(x).
\end{equation}
Equalities \eqref{eq:finite-section-method-1} and
\eqref{eq:finite-section-method-2} imply that condition (c) of
Theorem~\ref{th:main} for the sequence $\bP\bA\bP+\bQ$ coincides
with condition (c) of the present theorem.
\end{proof}
\subsection{Applicability of the finite section method to paired convolution operators}
To illustrate Theorem~\ref{th:finite-section-method}, we apply it
to the so-called \textit{paired convolution operator}
\begin{equation}\label{eq:paired-convolution}
A=W^0(a)\chi_-I+W^0(b)\chi_+I\quad (a,b\in [PC_p,SO_p]).
\end{equation}
For $g\in\cM_p$, the Wiener-Hopf operator on $L^p(\R_+)$ with
symbol $g$ is the operator $W(g)$ defined by
\[
W(g):L^p(\R_+)\to L^p(\R_+), \quad
\varphi\mapsto\chi_+W^0(g)\chi_+\varphi.
\]
Let $J$ denote the flip operator on the space $L^p(\R)$ defined by
$(J\varphi)(x)=\varphi(-x)$. For $g\in\cM_p$, put
$\widetilde{g}(x):=g(-x)$.
\begin{corollary}\label{co:fsm-paired}
The finite section method \eqref{eq:FSM} applies to the paired
convolution operator given by \eqref{eq:paired-convolution} if and
only if the following three conditions are fulfilled:
\begin{enumerate}
\item[{\rm(a)}]
the operator \eqref{eq:paired-convolution} is invertible on the
space $L^p(\R)$ and the Wiener-Hopf operators $W(a)$ and
$W(\widetilde{b})$ are invertible on the space $L^p(\R_+)$;

\item[{\rm(b)}]
for every $\eta\in\R$,
\[
a(\eta^\pm)\ne 0,\quad b(\eta^\pm)\ne 0,\quad
\operatorname{wind}\fC_p\left(1,\frac{a(\eta^-)}{a(\eta^+)},\frac{b(\eta^-)}{b(\eta^+)}\right)=0,
\]
where $\operatorname{wind}\fC_p(\cdot,\cdot,\cdot)$ is defined in
Subsection~{\rm\ref{subsec:SIO-piecewise-constant}};

\item[{\rm(c)}]
for every $\eta\in M_\infty(SO)$ and every $x\in\fL_p$,
\[
a_\eta(-\infty)x+a_\eta(+\infty)(1-x)\ne 0,
\quad
b_\eta(+\infty)x+b_\eta(-\infty)(1-x)\ne 0.
\]
\end{enumerate}
\end{corollary}
\begin{proof}
Let $\bA=(W^0(a)\chi_-I+W^0(b)\chi_+I)$. Taking into account
Proposition~\ref{pr:A-in-F}(b)--(c) we see that the operator
\[
\chi_+\fW_{-1}(\bA)\chi_+I+\chi_-I=\chi_+W^0(a)\chi_+I+\chi_-I
\]
is invertible on $L^p(\R)$ if and only if the Wiener-Hopf operator
$W(a)$ is invertible on $L^p(\R_+)$ and that the operator
\[
\chi_-\fW_1(\bA)\chi_-I+\chi_+I=\chi_-W^0(b)\chi_-I+\chi_+I
\]
is invertible on $L^p(\R)$ if and only if the operator
\[
J\chi_-W^0(b)\chi_-J=\chi_+JW^0(b)J\chi_+I= W(\widetilde{b})
\]
is invertible on $L^p(\R_+)$. This implies that condition (a) of
Theorem~\ref{th:finite-section-method} for the constant sequence
$\bA$ is nothing but condition (a) of the present corollary.

From Proposition~\ref{pr:A-in-Heta}(b)--(c) it follows that for every
$\eta\in\R$,
\[
\begin{split}
P_1\fH_\eta(\bA)P_1+Q_1 &=
P_1\big[(a(\eta^-)W^0(\chi_-)+a(\eta^+)W^0(\chi_+))\chi_-I
\\
&\quad+
(b(\eta^-)W^0(\chi_-)+b(\eta^+)W^0(\chi_+))\chi_+I\big]P_1+Q_1
\\
&= \chi_{(-1,1)}\big(a(\eta^-)P_\R+a(\eta^+)Q_\R\big)\chi_{(-1,0)}I
\\
&\quad +
\chi_{(-1,1)}\big(b(\eta^-)P_\R+b(\eta^+)Q_\R\big)\chi_{(0,1)}I+
\chi_{\R\setminus(-1,1)}I.
\end{split}
\]
Obviously this operator is invertible on $L^p(\R)$ if and only if
the operator
\[
P_{(-1,1)}(a(\eta^-)\chi_{(-1,0)}+b(\eta^-)\chi_{(0,1)})I+
Q_{(-1,1)}(a(\eta^+)\chi_{(-1,0)}+b(\eta^+)\chi_{(0,1)})I
\]
is invertible on $L^p(-1,1)$. By Lemma~\ref{le:SIO-piecewise-constant},
this operator is invertible on $L^p(-1,1)$ if and only if condition (b)
is fulfilled. This means that for $\bA$ condition (b) of
Theorem~\ref{th:finite-section-method} is equivalent to condition
(b) of the present corollary.

From formulas \eqref{eq:case-3-5}--\eqref{eq:case-3-8} it follows
that
\[
\begin{split}
[\fN_\eta^-(\bA)]_{11}(x) &=a_\eta(-\infty)x+a_\eta(+\infty)(1-x),
\\
[\fN_\eta^+(\bA)]_{11}(x) &=b_\eta(+\infty)x+b_\eta(-\infty)(1-x)
\end{split}
\]
for $x\in\fL_p$ and $\eta\in M_\infty(SO)$. This means that
condition (c) of Theorem~\ref{th:finite-section-method} specifies
to condition (c) of the present corollary.
\end{proof}

\end{document}